\documentclass[12pt, a4paper, reqno]{amsart}
\usepackage{eurosym}
\usepackage{amsmath}
\usepackage{amsfonts}
\usepackage{amssymb}
\usepackage{color}
\usepackage{comment,graphicx}
\usepackage[T1]{fontenc}
\usepackage{url}
\usepackage{ae}
\usepackage{mathtools}
\usepackage{esint}
\usepackage{latexsym}
\usepackage{amscd}
\usepackage{amsmath}
\usepackage{amsfonts}
\usepackage{amssymb}
\usepackage{color}
\usepackage{comment,graphicx}
\usepackage[T1]{fontenc}
\usepackage{url}
\usepackage{ae}

\def\phi{\varphi}
\def\rho{\varrho}
\def\epsilon{\varepsilon}
\numberwithin{equation}{section}
\theoremstyle{plain}
\newtheorem{theorem}[equation]{Theorem}
\newtheorem{lemma}[equation]{Lemma}

\newtheorem{corollary}[equation]{Corollary}
\theoremstyle{definition}
\newtheorem{definition}[equation]{Definition}

\theoremstyle{remark}
\newtheorem{remark}[equation]{Remark}

\renewcommand{\le}{\leqslant}
\renewcommand{\ge}{\geqslant}
\renewcommand{\leq}{\leqslant}
\renewcommand{\geq}{\geqslant}

\newcommand{\vertiii}[1]{{\left\vert\kern-0.25ex\left\vert\kern-0.25ex\left\vert \right\vert\kern-0.25ex\right\vert\kern-0.25ex\right\vert}}

\begin{document}
\title[Mixed-norm Herz-type Besov  spaces]{Mixed-norm
Herz-type Besov  spaces}
\author[D. Drihem]{Douadi Drihem}
\address{Douadi Drihem\\
Laboratory of Functional Analysis and Geometry of Spaces, Faculty of
Mathematics and Informatics, Department of Mathematics,\\
M'sila University, PO Box 166 Ichebelia, M'sila 28000, Algeria}
\email{douadidr@yahoo.fr, douadi.drihem@univ-msila.dz}
\thanks{ }
\date{\today }
\subjclass[2010]{ 46E35.}

\begin{abstract}
	In this paper, we introduce mixed-norm Herz-type Besov spaces and establish their $\varphi$-transform characterization in the sense of Frazier and Jawerth. We also prove Sobolev embedding theorems for these spaces. The resulting embeddings extend and improve several known embedding results for mixed-norm Besov spaces.
\end{abstract}

\keywords{Mixed-norm Herz, Besov space, Embedding.}
\maketitle

\section{Introduction}

Many families of function spaces play a central role in mathematical analysis, particularly in the study of partial differential equations. Classical examples include H\'{o}lder spaces, Bessel potential spaces, Besov spaces, and Triebel-Lizorkin spaces.

In recent years, considerable attention has been devoted to developing generalizations of these classical spaces and extending fundamental results to more general settings. One notable example is provided by Besov–Morrey spaces, which were introduced and studied in \cite{KoYa94}; see also \cite{Ma194,Ma203,Sata07,HWang09} for further developments. Another important example is given by Herz-type Besov spaces. These spaces extend the classical Besov  scales by replacing the underlying Lebesgue norm $L^p$ with a Herz-type norm, such as $\dot{K}_{p}^{\alpha,q}$ or $K_{p}^{\alpha,q}$; see \cite{XuYang05}.

Recently, Zhang, Yang, and Zhang \cite{ZYZ22} introduced a class of mixed-norm Herz spaces, providing a natural extension of mixed-norm Lebesgue spaces. They characterized the corresponding dual spaces, established a Riesz-Thorin interpolation theorem, and proved the boundedness of the Hardy-Littlewood maximal operator, as well as the Fefferman–Stein vector-valued maximal inequality, in this setting.

Furthermore, Zhang, Yang, and Zhang \cite{ZYZ22} proposed investigating Besov and Triebel-Lizorkin spaces, together with other scales of smoothness spaces, built upon mixed-norm Herz spaces.

Motivated by this suggestion, the main objective of the present paper is to develop the theory of Besov-type spaces associated with mixed-norm Herz spaces. More precisely, we introduce and investigate the spaces
$
\dot{E}_{\vec{p}}^{\vec{\alpha},\vec{q}}B_{\beta}^{s}.
$
These spaces provide a unified framework that incorporates the features of classical mixed-norm Besov spaces and Herz-type Besov spaces.

The paper is organized as follows. In Section 2, we recall several known results concerning mixed-norm Herz spaces and introduce the notation and preliminary material used throughout the paper. We also establish several technical tools, including an extension of the classical Plancherel-P\'{o}lya-Nikol'skij inequality to the mixed-norm Herz setting. In Section 3, using the Calder\'{o}n reproducing formula, we establish the $\varphi$-transform characterization of the spaces $\dot{E}_{\vec{p}}^{\vec{\alpha},\vec{q}}B_{\beta}^{s}$ in the sense of Frazier and Jawerth. Section 4 is devoted to basic embedding results for these spaces, while Section 5 establishes Sobolev-type embedding theorems.

\section{Mixed-norm\ Herz spaces}

The aim of this section is twofold. First, we recall the basic definitions and properties of mixed-norm Herz spaces. Second, we establish several technical results needed in the subsequent sections, including the boundedness of the Hardy-Littlewood maximal operator and a Plancherel-P\'{o}lya-Nikol'skij inequality in the mixed-norm Herz setting. These results provide essential tools for developing the main theory.

Throughout the paper, we adopt the following notation and conventions.
\subsection{Notation and conventions}

Throughout the paper, $\mathbb{R}^{n}$ denotes the $n$-dimensional Euclidean
space, $\mathbb{N}$ the set of positive integers, and $\mathbb{N}_{0}=%
\mathbb{N}\cup\{0\}$. The symbol $\mathbb{Z}$ stands for the set of all
integers.

For two non-negative quantities $f$ and $g$, we write $f\lesssim g$ if there
exists a positive constant $c$, independent of the relevant parameters, such
that $f\leq cg$. We write $f\approx g$ if both $f\lesssim g$ and $g\lesssim f
$ hold.

For $u>0$, let 
\begin{equation*}
C(u)=\left\{x\in\mathbb{R}^{n}:\frac{u}{2}\leq |x|<u\right\}. 
\end{equation*}
In particular, if $u=2^{k}$ with $k\in\mathbb{Z}$, we write $C_{k}=C(2^{k}).$

For a function $f$, $\operatorname{supp} f$ denotes its support. If
$E\subset\mathbb{R}^{n}$ is a measurable set, then $|E|$ denotes its
Lebesgue measure and $\chi_E$ its characteristic function.

Let $E\subset\mathbb{R}^{n}$ be measurable and let $0<p\leq\infty$. We
denote by $L^{p}(E)$ the space of all measurable functions $f:E\to\mathbb{C}$
such that 
\begin{equation*}
\|f\|_{L^{p}(E)} = \left(\int_{E}|f(x)|^{p}\,dx\right)^{1/p} <\infty, \qquad
0<p<\infty, 
\end{equation*}
and 
\begin{equation*}
\big\|f\big\|_{L^{\infty }(E)}=\underset{x\in E}{\text{ess-sup}}%
|f(x)|<\infty .
\end{equation*}
When $E=\mathbb{R}^{n}$, we simply write $L^{p}$ instead of $L^{p}(\mathbb{R}%
^{n})$ and $\|f\|_{p}$ instead of $\|f\|_{L^{p}(\mathbb{R}^{n})}$.

For $\vec p=(p_{1},\ldots,p_{n})$ with $p_{i}\in(0,\infty]$, $i\in\{1,\ldots,n\}$, we write $%
0<\vec p\leq\infty$. The mixed-norm Lebesgue space $L^{\vec p}(\mathbb{R}%
^{n})$ consists of all measurable functions $f:\mathbb{R}^{n}\to\mathbb{C}$
satisfying 
\begin{equation*}
\|f\|_{L^{\vec p}(\mathbb{R}^{n})} = \Bigg(
\int_{\mathbb{R}} \Bigg(
\cdots \Bigg(
\int_{\mathbb{R}} |f(x_{1},\ldots,x_{n})|^{p_{1}} \,dx_{1} \Bigg)%
^{p_{2}/p_{1}} \cdots \Bigg)^{p_{n}/p_{n-1}} dx_{n} \Bigg)^{1/p_{n}}
<\infty, 
\end{equation*}
with the usual modification of replacing the corresponding integral by the
essential supremum whenever $p_{j}=\infty$. Equipped with this quasi-norm, $%
L^{\vec p}(\mathbb{R}^{n})$ is a quasi-Banach space and becomes a Banach
space whenever $\min\{p_{1},\ldots,p_{n}\}\geq 1$. If $p_{1}=\cdots=p_{n}=p$%
, then 
\begin{equation*}
L^{\vec p}(\mathbb{R}^{n})=L^{p}(\mathbb{R}^{n}). 
\end{equation*}
For further details on mixed-norm Lebesgue spaces, we refer to \cite{BP61}.

We denote by $\mathcal{S}(\mathbb{R}^{n})$ the Schwartz space of rapidly
decreasing smooth functions on $\mathbb{R}^{n}$. Its dual space, denoted by $\mathcal{S}^{\prime }(\mathbb{R}^{n})$, is the
space of tempered distributions.

The Fourier transform of $f\in\mathcal{S}(\mathbb{R}^{n})$ is defined by 
\begin{equation*}
\mathcal{F}f(\xi) = (2\pi)^{-n/2} \int_{\mathbb{R}^{n}}
e^{-ix\cdot\xi}f(x)\,dx, \qquad \xi\in\mathbb{R}^{n}, 
\end{equation*}
and its inverse is denoted by $\mathcal{F}^{-1}f$. Both operators extend to $%
\mathcal{S}^{\prime }(\mathbb{R}^{n})$ in the usual way.

For $v\in\mathbb{N}_{0}$ and $m\in\mathbb{Z}^{n}$, let 
\begin{equation*}
Q_{v,m} = 2^{-v}\bigl(\lbrack 0,1)^{n}+m\bigr)
\end{equation*}
denote the dyadic cube of side length $2^{-v}$. We write $%
\chi_{v,m}=\chi_{Q_{v,m}}.$

Let $N,R>0$. We define 
\begin{equation*}
\eta_{R,N}(x) = R^{n}\prod_{i=1}^{n}(1+R|x_{i}|)^{-N}, \qquad
x=(x_{1},\ldots,x_{n})\in\mathbb{R}^{n}, 
\end{equation*}
and, in particular, 
\begin{equation*}
\eta_{i,R,N}(x_i) = R(1+R|x_i|)^{-N}, \qquad x_i\in\mathbb{R}, \quad
i\in \{1,\dots,n\}. 
\end{equation*}

For $0<\vec p\leq\infty$, $\vec\alpha=(\alpha_{1},\ldots,\alpha_{n})\in%
\mathbb{R}^{n}$, and $\vec\alpha_{j}=(\alpha_{1}^{j},\ldots,\alpha_{n}^{j})%
\in\mathbb{R}^{n}$, $j\in\{1,2\}$, we use the abbreviations 
\begin{equation*}
p_{-} = \min\{p_{1},\ldots,p_{n}\}, \qquad \frac{1}{\mathbf{p}} =
\sum_{i=1}^{n}\frac{1}{p_{i}}, 
\end{equation*}
and 
\begin{equation*}
\boldsymbol{\alpha}_{j}=\sum_{i=1}^{n}\alpha _{i}^{j}, \qquad \boldsymbol{\alpha}=\sum_{i=1}^{n}\alpha _{i},. 
\end{equation*}

Finally, $\mathfrak{M}(\mathbb{R}^{n})$ denotes the collection of all
measurable functions on $\mathbb{R}^{n}$. Any additional notation will be
introduced as needed.

\subsection{Definition and some basic properties}

The main purpose of this subsection is to present some fundamental
properties of mixed-norm Herz spaces. For $i\in \{1,\dots,n\}$ and $k_{i}\in 
\mathbb{Z}$, let 
\begin{equation*}
R_{k_{i}}=\{x_{i}\in \mathbb{R}:2^{k_{i}-1}\leq |x_{i}|<2^{k_{i}}\} \quad 
\text{and} \quad \chi_{k_{i}}=\chi_{R_{k_{i}}}.
\end{equation*}

\begin{definition}
Let $0<\vec{p},\vec{q}\leq \infty$ and $\vec{\alpha}=(\alpha_{1},\dots,%
\alpha_{n})\in \mathbb{R}^{n}$. The mixed-norm Herz space $\dot{E}_{\vec{p}%
}^{\vec{\alpha},\vec{q}}(\mathbb{R}^{n})$ is defined as the set of all
functions $f\in \mathfrak{M}(\mathbb{R}^{n})$ such that  
\begin{equation*}
\big\|f\big\|_{\dot{E}_{\vec{p}}^{\vec{\alpha},\vec{q}}(\mathbb{R}^{n})}= %
\big\|\cdot \cdot \cdot \big\|f\big\|_{\dot{K}_{p_{1}}^{\alpha
_{1},q_{1}}}\cdot \cdot \cdot \big\|_{\dot{K}_{p_{n}}^{\alpha
_{n},q_{n}}}<\infty ,
\end{equation*}
where  
\begin{equation*}
\|f\|_{\dot{K}_{p_{i}}^{\alpha_{i},q_{i}}} = \Big(\sum_{k_{i}\in \mathbb{Z}}
2^{k_{i}\alpha_{i}q_{i}} \|f\,\chi_{k_{i}}\|_{L^{{p_{i}}}(\mathbb R)}^{q_{i}}\Big)^{1/q_{i}},
\quad i\in \{1,\dots,n\}.
\end{equation*}
\end{definition}

\begin{remark}
When $p_{i}=q_{i}=p$ for all $i\in \{1,\dots,n\}$ and $\vec{\alpha}=\vec{0}=(0,\dots,%
0)$, the mixed-norm Herz spaces $\dot{E}_{\vec{p}}^{\vec{0},\vec{q}}(%
\mathbb{R}^{n})$ coincide with the Lebesgue space $L^{p}(\mathbb{R}^{n})$.

Moreover, when $p_{i}=q_{i}$ for all $i\in \{1,\dots,n\}$ and $\vec{\alpha}=%
\vec{0}$, they coincide with the mixed Lebesgue space $L^{\vec{p}}(%
\mathbb{R}^{n})$.

More details on mixed-norm Herz spaces can be found in \cite{ZYZ22}.
\end{remark}

The Hardy-Littlewood maximal operator $M$ is defined for locally integrable
functions by 
\begin{equation*}
Mf(x)=\sup_{r>0}\frac{1}{|B(x,r)|}\int_{B(x,r)}|f(y)|\,dy.
\end{equation*}

When the maximal operator is applied only in the variable $x_k$, using the
decomposition $x=(x^{\prime },x_k,x^{\prime \prime })$, we write 
\begin{equation*}
M_k f(x_1,\dots,x_n)=(Mf(x_1,\dots,x_n))(x_k).
\end{equation*}

Furthermore, for $t>0$, the iterated maximal operator $\mathcal{M}_t$ is
defined by 
\begin{equation*}
\mathcal{M}_{t}(f)(x)=\big(M_{n}\cdots M_{1}(|f|^{t})(x)\big)^{1/t}.
\end{equation*}

The proof of the main results of this paper is based on the following lemma;
see \cite{ZYZ22}.

\begin{lemma}
\label{Maximal-Inq}  Let $0<\beta \leq \infty$, $0<\vec{p},\vec{q}<\infty$,
and $\vec{\alpha}=(\alpha_{1},\dots,\alpha_{n})\in \mathbb{R}^{n}$ satisfy,
for each $i\in \{1,\dots,n\}$,  
\begin{equation*}
-\frac{1}{p_{i}}<\alpha_{i}<1-\frac{1}{p_{i}}, \quad \text{and} \quad
0<t<\min(\beta,p_{-},q_{-}).
\end{equation*}
Then there exists a constant $C>0$ such that, for any $f\in \dot{E}_{\vec{p}}^{\vec{\alpha},\vec{q}}(\mathbb{R}^{n})$%
,  
\begin{equation*}
\big\|\mathcal{M}_{t}(f)\big\|_{\dot{E}_{\vec{p}}^{\vec{\alpha},\vec{q}}(\mathbb{R}^{n})}
\leq C \big\|f
\big\|_{\dot{E}_{\vec{p}}^{\vec{\alpha},\vec{q}}(\mathbb{R}^{n})}.
\end{equation*}
\end{lemma}

Let $(y_{1},\dots,y_{n})\in \mathbb{R}^{n}$ and $i\in \{1,\dots,n\}$. We set 
\begin{equation*}
\hat{y}_{i}=(y_{i},\dots,y_{n})\in \mathbb{R}^{n-i+1}, \quad \text{and}
\quad \check{y}_{i}=(y_{1},\dots,y_{i})\in \mathbb{R}^{i}.
\end{equation*}

Let $i\in{1,\dots,n-1}$, $d>0$, $(y_2,\dots,y_n)\in\mathbb{R}^{n-1}$, $0<\vec p,\vec q\leq\infty$, and $\vec\alpha\in\mathbb{R}^n$. For $f\in\dot{E}_{\vec p}^{\vec\alpha,\vec q}(\mathbb{R}^n))$, we recall that
\begin{equation*}
\big\|\cdot \cdot \cdot \big\|f\big\|_{\dot{K}_{p_{1}}^{\alpha
		_{1},q_{1}}}\cdot \cdot \cdot \big\|_{\dot{K}_{p_{i}}^{\alpha _{i},q_{i}}}(%
\hat{y}_{i+1}),
\end{equation*}
is a function of $n-i$ variables, namely, the last $n-i$ variables of $f$.

\subsection{Plancherel-P\'{o}lya-Nikolskij inequality}

The classical Plancherel-P\'{o}lya-Nikol'skij inequality (see \cite[1.3.2/5,
Remark 1.4.1/4]{T83}) asserts that 
\begin{equation*}
\|f\|_{q} \leq c\,R^{\,n\left(\frac1p-\frac1q\right)}\|f\|_{p}, 
\end{equation*}
for all $0<p\leq q\leq \infty$, $R>0$, and every $f\in L^{p}(\mathbb{R}%
^{n})\cap \mathcal{S}^{\prime }(\mathbb{R}^{n})$ satisfying 
\begin{equation*}
\operatorname{supp}\mathcal{F}f
\subset \{\xi\in\mathbb{R}^{n}: |\xi|\le R\}.
\end{equation*}
The constant $c>0$ is independent of $R$. This inequality plays a
fundamental role in the theory of function spaces and partial differential
equations. The present subsection is devoted to extending this result to
mixed-norm Herz spaces. The following lemma is essential for our analysis;
see \cite[Lemma A.6]{DHR}.

\begin{lemma}
\label{r-trick}  Let $r>0$, $j\in\mathbb{N}_{0}$, and $m>n$. Then there
exists a constant  $c=c(r,m,n)>0$ such that for every  $f\in\mathcal{S}%
^{\prime }(\mathbb{R}^{n})$ with  
\begin{equation*}
\operatorname{supp}\mathcal{F}f
\subset
\left\{\xi\in\mathbb{R}^{n}:|\xi|\le 2^{j+1}\right\}.
\end{equation*}
we have  
\begin{equation*}
|f(x)|  \leq  c\Bigl(\eta_{j,\frac{m}{n}}*|g|^{r}(x)\Bigr)^{1/r},  \qquad
x\in\mathbb{R}^{n}.  
\end{equation*}
\end{lemma}

Let 
\begin{equation*}
I_{1,k}=[-2^{k},2^{k}], \qquad I_{2,k}=\Bigl[\frac{2^{k-1}}{\sqrt n},\,2^{k}%
\Bigr], \qquad k\in\mathbb{Z}, 
\end{equation*}
and 
\begin{equation*}
I_{3,k} = \Bigl[-2^{k},-\frac{2^{k-1}}{\sqrt n}\Bigr], \qquad I_{4,k} = %
\Bigl(-\frac{2^{k-1}}{\sqrt n},\,\frac{2^{k-1}}{\sqrt n}\Bigr), \qquad k\in%
\mathbb{Z}. 
\end{equation*}
Define 
\begin{equation*}
J_{k} = \bigcup_{l=1}^{n-1}V_{l,k}\,\cup\,V_{k}, \qquad k\in\mathbb{Z}, 
\end{equation*}
where 
\begin{equation*}
V_{k} = \bigl(I_{2,k}\times(I_{4,k})^{n-1}\bigr)
\cup \bigl(I_{3,k}\times(I_{4,k})^{n-1}\bigr), 
\end{equation*}
and 
\begin{equation*}
V_{l,k} = V_{l,k}^{1}\cup V_{l,k}^{2}, \qquad k\in\mathbb{Z}, \quad
l\in\{1,\ldots,n-1\}, 
\end{equation*}
with 
\begin{equation*}
V_{l,k}^{1} = (I_{1,k})^{\,n-l}\times I_{2,k}\times (I_{4,k})^{\,l-1}, 
\end{equation*}
and 
\begin{equation*}
V_{l,k}^{2} = (I_{1,k})^{\,n-l}\times I_{3,k}\times (I_{4,k})^{\,l-1}. 
\end{equation*}
In the special case $l=1$, we set 
\begin{equation*}
V_{1,k}^{1} = (I_{1,k})^{\,n-1}\times I_{2,k}, \qquad V_{1,k}^{2} =
(I_{1,k})^{\,n-1}\times I_{3,k}. 
\end{equation*}

The following result is taken from \cite{Dr-Sobolev}.

\begin{lemma}
\label{decomposition1}  Let $k\in\mathbb{Z}$. Then  
\begin{equation*}
C_{k}\subset J_{k}\subset \widetilde C_{k},  
\end{equation*}
where  
\begin{equation*}
\widetilde C_{k}  =  \left\{  x\in\mathbb{R}^{n}  :  \frac{2^{k-3}}{\sqrt n}
\le |x|  \le  \sqrt n\,2^{k+4}  \right\}.  
\end{equation*}
\end{lemma}

Let $h,m\in \mathbb{N}$ be such that $1\leq h\leq n-m+1$. Let $\vec{\alpha}%
=(\alpha _{1},\ldots ,\alpha _{n})\in \mathbb{R}^{n}$ and $0<\vec{p}\leq
\infty $. We set 
\begin{equation*}
\frac{1}{\mathbf{p}%
_{h,h+m-1}}=\sum_{i=h}^{h+m-1}\frac{1}{p_{i}},
\end{equation*}%
and 
\begin{equation*}
\boldsymbol{\alpha}_{h,h+m-1}=\sum_{i=h}^{h+m-1}\alpha _{i}.
\end{equation*}%
In particular, if $h=1$ and $m=n$, we write 
\begin{equation*}
\mathbf{p}_{1,n}=\mathbf{p},\qquad \boldsymbol{\alpha}_{1,n}=\boldsymbol{\alpha}.
\end{equation*}
Let $i\in \{1,...,n\}$. We set $\digamma _{i}( p_{i},\alpha_{i})=\min \Big(p_{i},\frac{1}{\frac{%
		1}{p_{i}}+\alpha _{i}}\Big).$

Let $h,m\in\mathbb{N}$ with $1\le h\le n-m+1$. Throughout the next lemma and
remark, we define 
\begin{equation*}
\widetilde{\eta}_{h,R,N}(\widetilde z) = R^m \prod_{i=h}^{h+m-1}
(1+R|z_i|)^{-N}, \qquad \widetilde z=(z_h,\ldots,z_{h+m-1})\in\mathbb{R}^m . 
\end{equation*}

\begin{lemma}
\label{key1 copy(2)}  Let $R>0$ and let $h,m,n\in\mathbb{N}$ satisfy $1\le
h\le n-m+1$. Let $l\in\{1,4\}$, $v\in\mathbb{Z}$, and $x\in\mathbb{R}^n$.
Assume that  $0<\vec p\le\infty$, $\vec{\alpha}=(\alpha_1,\ldots,\alpha_n)\in%
\mathbb{R}^n$, and  $f\in \dot E_{\vec p}^{\vec\alpha,\infty}(\mathbb{R}^n)$%
.  Then  
\begin{align}
&\int_{(I_{l,v})^m} |f(y)|^d \widetilde{\eta}_{h,R,N}(\widetilde x-\widetilde
y) \,d\widetilde y  \notag \\
&\lesssim R^m\, 2^{-v\left( \frac1{\mathbf{p}_{h,h+m-1}} +\boldsymbol{\alpha}%
_{h,h+m-1} \right)d+mv} \, \big\| \cdots \|f\|_{\dot
K_{p_h}^{\alpha_h,\infty}} \cdots \big\|_{\dot
K_{p_{h+m-1}}^{\boldsymbol{\alpha}_{h+m-1},\infty}}^d,  \label{constant}
\end{align}
where  $0<d<\min_{h\le i\le h+m-1} \digamma _{i}( p_{i},\alpha_{i})  
$ and the implicit constant is independent of $R$, $v$, and $x$. Here  
\begin{equation*}
\widetilde y=(y_h,\ldots,y_{h+m-1})\in\mathbb{R}^m,\qquad  \widetilde
x=(x_h,\ldots,x_{h+m-1})\in\mathbb{R}^m.  
\end{equation*}
\end{lemma}

\begin{proof}
Since the cases $l=1$ and $l=4$ are analogous, it suffices to consider  $l=1$%
. We proceed by induction on $m$.

\medskip

\noindent  \textit{Step 1. The case $m=1$.} We have  
\begin{align}
\int_{I_{1,v}} |f(y)|^d \eta_{h,R,N}(x_h-y_h)\,dy_h &\le \sum_{j=0}^{\infty}
\int_{R_{(v-j+1)_h}} |f(y)|^d \eta_{h,R,N}(x_h-y_h)\,dy_h  \notag \\
&= \sum_{j=0}^{\infty} V_{j,v}^{1}(x_h),  \label{aux3bernstien2}
\end{align}
where  
\begin{equation*}
R_{(v-j+1)_h}  =  \{t\in\mathbb{R}:  2^{\,v-j}\le |t|<2^{\,v-j+1}\}.  
\end{equation*}

Applying H\"{o}lder's inequality with  
\begin{equation*}
\frac1d=\frac1{p_h}  +\left(\frac1d-\frac1{p_h}\right),  
\end{equation*}
we obtain  
\begin{align}
\sum_{j=0}^{\infty}V_{j,v}^{1}(x_h) &\lesssim R \sum_{k_h=-\infty}^{v+1} %
\big\| f(y_1,\ldots,y_{h-1},\cdot_h,y_{h+1},\ldots,y_n)
\chi_{R_{k_h}}(\cdot_h) \big\|_{L^{{d}}(\mathbb R)}^d  \notag \\
&\lesssim R \sum_{k_h=-\infty}^{v+1} 2^{d k_h(\frac1d-\frac1{p_h})} \big\| %
f(y_1,\ldots,y_{h-1},\cdot_h,y_{h+1},\ldots,y_n) \chi_{R_{k_h}}(\cdot_h) %
\big\|_{L^{{p_{h}}}(\mathbb R)}^{d}  \notag \\
&\lesssim R\, 2^{-v\left(\frac1{p_h}+\alpha_h\right)d+v} \, \|f\|_{\dot
K_{p_h}^{\alpha_h,\infty}}^d,  \label{aux3bernstien1}
\end{align}
where we have used the assumption $
\frac1d>\frac1{p_h}+\alpha_h. $

Combining \eqref{aux3bernstien2} and \eqref{aux3bernstien1}, we conclude
that  
\begin{equation*}
\int_{I_{1,v}}  |f(y)|^d  \eta_{h,R,N}(x_h-y_h)\,dy_h  \lesssim  R\, 
2^{-v\left(\frac1{p_h}+\alpha_h\right)d+v}  \,  \|f\|_{\dot
K_{p_h}^{\alpha_h,\infty}}^d.  
\end{equation*}

\medskip

\noindent  \textit{Step 2. Induction step.}

Assume that the assertion holds for some $1\leq m<n$. Define  
\begin{equation*}
\widehat{\eta}_{h,R,N}(\widehat z)  =  R^{m+1}  \prod_{i=h}^{h+m} 
(1+R|z_i|)^{-N},  \qquad  \widehat z=(z_h,\ldots,z_{h+m})  \in\mathbb{R}%
^{m+1}.  
\end{equation*}

Then  
\begin{equation*}
\int_{(I_{l,v})^{m+1}}  |f(y)|^d  \widehat{\eta}_{h,R,N}(\widehat x-\widehat
y)  \,d\widehat y  
\end{equation*}
can be written as  
\begin{equation*}
\int_{I_{l,v}}  \eta_{h+m,R,N}(x_{h+m}-y_{h+m})  \left(  \int_{(I_{l,v})^m} 
|f(y)|^d  \widetilde{\eta}_{h,R,N}(\widetilde x-\widetilde y)  \,d\widetilde y
\right)  dy_{h+m}.  
\end{equation*}

By the induction hypothesis, this is bounded by  
\begin{align*}
& R^m 2^{-v\left( \frac1{\mathbf{p}_{h,h+m-1}} +\boldsymbol{\alpha}_{h,h+m-1}
\right)d+mv} \\
&\qquad\times \int_{I_{l,v}} \big\|\cdots \|f\|_{\dot
K_{p_h}^{\alpha_h,\infty}} \cdots \big\|_{\dot
K_{p_{h+m-1}}^{\alpha_{h+m-1},\infty}}^d \, \eta_{h+m,R,N}(x_{h+m}-y_{h+m})
\,dy_{h+m}.
\end{align*}

Applying Step~1 to the variable $y_{h+m}$, we obtain that the last term is
bounded by  
\begin{align*}
&c R^{m+1} 2^{-v\left( \frac1{\mathbf{p}_{h,h+m-1}} +\boldsymbol{\alpha}%
_{h,h+m-1} \right)d+mv} 2^{-v\left( \frac1{p_{h+m}} +\alpha_{h+m} \right)d+v}
\\
&\qquad\times \big\|\cdots \|f\|_{\dot K_{p_h}^{\alpha_h,\infty}} \cdots %
\big\|_{\dot K_{p_{h+m}}^{\alpha_{h+m},\infty}}^d \\
&=c R^{m+1} 2^{-v\left( \frac1{\mathbf{p}_{h,h+m}} +\boldsymbol{\alpha}_{h,h+m}
\right)d+(m+1)v} \big\|\cdots \|f\|_{\dot K_{p_h}^{\alpha_h,\infty}} \cdots %
\big\|_{\dot K_{p_{h+m}}^{\alpha_{h+m},\infty}}^d.
\end{align*}

This completes the induction and hence the proof.
\end{proof}

\begin{remark}
\label{Key2}  \textnormal{(i)} Under the same assumptions as in Lemma \ref%
{key1 copy(2)}, we similarly obtain  
\begin{align*}
& \int_{(I_{l,v})^{m}} \big\|\cdots \big\|f\big\|_{\dot{K}%
_{p_{1}}^{\alpha_{1},\infty}} \cdots \big\|_{\dot{K}_{p_{h-1}}^{%
\alpha_{h-1},\infty}}^{d} \,\tilde{\eta}_{h,R,N}(\tilde{x}-\tilde{y}) \,d%
\tilde{y} \\
& \qquad\lesssim R^{m}\, 2^{-v\left(\frac{1}{\mathbf{p}_{h,h+m-1}} +\boldsymbol{\alpha}_{h,h+m-1}\right)d+mv} \big\|\cdots \big\|f\big\|_{\dot{K}%
_{p_{1}}^{\alpha_{1},\infty}} \cdots \big\|_{\dot{K}_{p_{h+m-1}}^{%
\alpha_{h+m-1},\infty}}^{d}.
\end{align*}

\medskip

\noindent \textnormal{(ii)} Note that, in \eqref{constant}, the implicit
constant may depend on  
\begin{equation*}
(y_{1},\ldots,y_{h-1},y_{h+m},\ldots,y_{n}).  
\end{equation*}
\end{remark}

\begin{lemma}
\label{Key-est1}  \textit{Let } $\vec{\alpha}=(\alpha_1,\ldots,\alpha_n)\in%
\mathbb{R}^n$, $0<\vec p,\vec
q\leq\infty$, \textit{and }  $R\geq H>0$. \textit{Then there exists a
constant } $c>0$,  \textit{independent of } $R$ \textit{and } $H$, \textit{%
such that for every }  $f\in \dot{E}_{\vec p}^{\vec\alpha,\vec q}(\mathbb{R}%
^{n})  \cap \mathcal{S}^{\prime }(\mathbb{R}^{n})$  \textit{with}  
\begin{equation*}
\operatorname{supp}\mathcal{F}f
\subset
\bigl\{\xi\in\mathbb{R}^{n}: |\xi|\le R\bigr\},
\end{equation*}
\textit{one has}  
\begin{equation*}
\sup_{x\in B(0,\frac{1}{H})}|f(x)|  \le  c\Big(\frac{R}{H}\Big)^{\frac{n}{d}}  H^{%
\frac{1}{\mathbf{p}}+\boldsymbol{\alpha}}  \|f\|_{\dot{E}_{\vec
p}^{\vec\alpha,\vec q}(\mathbb{R}^{n})},  
\end{equation*}
\textit{for every}  
$0<d<\min_{1\le i\le n} \digamma _{i}( p_{i},\alpha_{i}) $.

\end{lemma}

\begin{proof}
Let $v \in \mathbb{Z}$ be such that $2^{v-2} \leq H^{-1} < 2^{v-1}$. By
Lemma~\ref{r-trick}, we have for any $d,R>0$, $N>\frac{n}{d}$, and any $x
\in B(0,\frac{1}{H})$, 
\begin{equation*}
|f(x)| \lesssim \Bigg(\int_{\mathbb{R}^{n}} |f(y)|^{d}\, \eta_{R,dN}(x-y)\,
dy \Bigg)^{\frac{1}{d}} \lesssim I_{1} + I_{2},
\end{equation*}
where 
\begin{equation*}
I_{1} = \Bigg(\int_{\overline{B(0,\frac{4}{H})}} |f(y)|^{d}\,
\eta_{R,dN}(x-y)\, dy \Bigg)^{\frac{1}{d}}
\end{equation*}
and 
\begin{equation*}
I_{2} = \Bigg(\int_{\mathbb{R}^{n} \setminus \overline{B(0,\frac{4}{H})}}
|f(y)|^{d}\, \eta_{R,dN}(x-y)\, dy \Bigg)^{\frac{1}{d}}.
\end{equation*}

\textit{Estimate of $I_{1}$.} Observe that $\overline{B(0,\frac{4}{H})}
\subset (I_{1,v+1})^{n}$. Then Lemma~\ref{key1 copy(2)}, with $m=n$ and $h=1$%
, yields 
\begin{align*}
(I_{1})^{d} & \leq \int_{(I_{1,v+1})^{n}} |f(y)|^{d}\, \eta_{R,dN}(x-y)\, dy
\\
& \lesssim R^{n}\, 2^{-v\left(\frac{1}{\mathbf{p}}+\boldsymbol{\alpha}\right)d +
nv} \|f\|_{\dot{K}_{\vec{p}}^{\vec{\alpha},\infty}}^{d} \\
& \lesssim \left(\frac{R}{H}\right)^{n} H^{\left(\frac{1}{\mathbf{p}}+%
\boldsymbol{\alpha}\right)d} \|f\|_{\dot{K}_{\vec{p}}^{\vec{\alpha},\infty}}^{d}.
\end{align*}
\textit{Estimate of }$I_{2}$. Let $x\in B\!\left(0,\frac{1}{H}\right)$. We
have the following decomposition 
\begin{align}
(I_{2})^{d} &\leq \int_{\mathbb{R}^{n}\setminus \overline{B(0,2^{v})}}
|f(y)|^{d}\,\eta_{R,dN}(x-y)\,dy  \notag \\
&\leq \sum_{j=0}^{\infty} \int_{C_{j+v+1}} |f(y)|^{d}\,\eta_{R,dN}(x-y)\,dy .
\label{int}
\end{align}

By Lemma \ref{decomposition1}, we have 
\begin{equation*}
C_{j+v+1}\subset J_{j+v+1}\subset \tilde{C}_{j+v+1}. 
\end{equation*}

The integral in \eqref{int} can be estimated from above by 
\begin{equation*}
\tilde{J}_{j,v}^{1}(x)+\tilde{J}_{j,v}^{2}(x) +\sum_{l=1}^{n-1}\big(%
J_{j,v,l}^{1}(x)+J_{j,v,l}^{2}(x)\big),
\end{equation*}
where 
\begin{align*}
J_{j,v,l}^{1}(x) &=\int_{V_{l,j+v+1}^{1}} |f(y)|^{d}\,\eta_{R,dN}(x-y)\,dy,
\quad l\in\{1,\dots,n-1\}, \\
J_{j,v,l}^{2}(x) &=\int_{V_{l,j+v+1}^{2}} |f(y)|^{d}\,\eta_{R,dN}(x-y)\,dy,
\quad l\in\{1,\dots,n-1\}, \\
\tilde{J}_{j,v}^{1}(x) &=\int_{I_{2,j+v+1}\times (I_{4,j+v+1})^{n-1}}
|f(y)|^{d}\,\eta_{R,dN}(x-y)\,dy, \\
\tilde{J}_{j,v}^{2}(x) &=\int_{I_{3,j+v+1}\times (I_{4,j+v+1})^{n-1}}
|f(y)|^{d}\,\eta_{R,dN}(x-y)\,dy.
\end{align*}

It suffices to estimate $J_{j,v,l}^{1},
J_{j,v,l}^{2}$ for $l\in\{1,\dots,n-1\}$ and $\tilde{J}_{j,v}^{1}, \tilde{J}%
_{j,v}^{2}$. Clearly, it is enough to estimate $J_{j,v,l}^{1}$ and $\tilde{J}%
_{j,v}^{1}$.

\textit{Estimate of }$J_{j,v,l}^{1}$, $l\in\{1,\dots,n-1\}$.

\textit{Case 1.} Let $j\in\mathbb{N}$ such that $2^{j}>\sqrt{n}$. Then for
any $y_{n-l+1}\in I_{2,j+v+1}$ and $x\in B(0,2^{v-1})$, we have 
\begin{equation*}
|x_{n-l+1}-y_{n-l+1}| >\frac{1}{\sqrt{n}}2^{j+v-1}, 
\end{equation*}
hence for any $d>0$, $N>\frac{n}{d}$, 
\begin{equation}
\eta_{n-l+1,R,dN}(x_{n-l+1}-y_{n-l+1}) \leq R(2^{j+v}R)^{-dN} \leq 2^{-jdN}R.
\label{est-neu}
\end{equation}

Using Lemma \ref{key1 copy(2)}, we obtain 
\begin{align*}
&\int_{(I_{1,j+v+1})^{n-l}} |f(y)|^{d}\, \tilde{\eta}_{R,dN}(\tilde{x}-%
\tilde{y})\,d\tilde{y} \\
&\quad \lesssim 2^{(n-l+\frac{1}{\mathbf{p}_{1,n-l}}+\boldsymbol{\alpha}_{1,n-l})j}R^{\,n-l} H^{\left(\frac{1}{\mathbf{p}_{1,n-l}}+\boldsymbol{\alpha}_{1,n-l}\right)d-n+l} \big\|\cdot \cdot \cdot \big\|f\big\|_{\dot{K}%
_{p_{1}}^{\alpha _{1},\infty }}\cdot \cdot \cdot \big\|_{\dot{K}%
_{p_{n-l}}^{\alpha _{n-l},\infty }}^{d},
\end{align*}
where 
\begin{equation*}
\tilde{\eta}_{R,dN}(\tilde{z}) =R^{n-l}\prod_{i=1}^{n-l}(1+R|z_i|)^{-dN},
\quad \tilde{z}=(z_1,\ldots,z_{n-l}) \in\mathbb{R}^{n-l}. 
\end{equation*}

Therefore, 
\begin{align*}
&\int_{I_{2,j+v+1}}\int_{(I_{1,j+v+1})^{n-l}} |f(y)|^{d}\tilde{\eta}_{R,dN}(%
\tilde{x}-\tilde{y}) \eta_{n-l+1,R,dN}(x_{n-l+1}-y_{n-l+1}) \,d\tilde{y}%
\,dy_{n-l+1} \\
&\quad \lesssim 2^{(n-l+\frac{1}{\mathbf{p}_{1,n-l}}+\boldsymbol{\alpha}_{1,n-l})j} \left(\frac{R}{H}\right)^{n-l} H^{\left(\frac{1}{\mathbf{p}%
_{1,n-l}}+\boldsymbol{\alpha}_{1,n-l}\right)d}
\\
& \times
 \int_{I_{2,j+v+1}}  \big\|\cdot \cdot \cdot \big\|f\big\|_{\dot{K}%
_{p_{1}}^{\alpha _{1},\infty }}\cdot \cdot \cdot \big\|_{\dot{K}%
_{p_{n-l}}^{\alpha _{n-l},\infty }}^{d},
\eta_{n-l+1,R,dN}(x_{n-l+1}-y_{n-l+1})\,dy_{n-l+1},
\end{align*}
which is bounded by, using \eqref{est-neu} and H\"{o}lder's inequality  with 
  $\frac1d=\frac1{p_n-l+1}  +\left(\frac1d-\frac1{p_n-l+1}\right),$ 
\begin{align*}
&c 2^{-j(l-n+dN-1+\frac{d}{\mathbf{p}_{1,n-l+1}}+d\boldsymbol{\alpha}
	_{1,n-l+1})}
\\
&\times \left(\frac{R}{H}%
\right)^{n-l+1} H^{\left(\frac{1}{\mathbf{p}_{1,n-l+1}}+\boldsymbol{\alpha}%
_{1,n-l+1}\right)d}
\big\|\cdot \cdot \cdot \big\|f\big\|_{\dot{K}%
	_{p_{1}}^{\alpha _{1},\infty }}\cdot \cdot \cdot \big\|_{\dot{K}%
	_{p_{n-l+1}}^{\alpha _{n-l+1},\infty }}^{d}.
\end{align*}

Again by Lemma \ref{key1 copy(2)}, we obtain%
\begin{align*}
& \int_{(I_{4,j+v+1})^{l-1}}\big\|\cdot \cdot \cdot \big\|f\big\|_{\dot{K}%
	_{p_{1}}^{\alpha _{1},\infty }}\cdot \cdot \cdot \big\|_{\dot{K}%
	_{p_{n-l+1}}^{\alpha _{n-l+1},\infty }}^{d}\check{\eta}_{R,dN}(\check{x}-%
\check{y})d\check{y} \\
& \lesssim 2^{(l-1+\frac{1}{\mathbf{p}_{n-l+2,n}}+\boldsymbol{\alpha}_{n-l+2,n})j}R^{l-1}\text{ }2^{-v(\frac{1}{\mathbf{p}_{n-l+2,n}}+\boldsymbol{\alpha}_{n-l+2,n})d+(l-1)v}\big\|f\big\|_{\dot{K}_{\vec{p}}^{\vec{\alpha}%
		,\infty }}^{d},
\end{align*}%
where 
\begin{equation*}
\check{\eta}_{R,dN}(\check{z})=R^{l-1}\prod_{i=n-l+2}^{n}(1+R|z_{i}|)^{-dN},%
\quad \check{z}=(z_{n-l+2},...,z_{n})\in \mathbb{R}^{l-1}.
\end{equation*}%

Consequently, 
\begin{equation}
J_{j,v,l}^{1} \lesssim 2^{-j(dN-n+\frac{d}{\mathbf{p}}+d\boldsymbol{\alpha})}
\left(\frac{R}{H}\right)^{n} H^{\left(\frac{1}{\mathbf{p}}+\boldsymbol{\alpha}%
\right)d} \|f\|_{\dot{E}_{\vec{p}}^{\vec{\alpha},\vec{\theta}}(\mathbb{R}%
^n)}^{d},  \label{case1}
\end{equation}
for $2^{j}>\sqrt{n}$.

\textit{Case 2.} If $2^{j}\leq \sqrt{n}$, then observe that
\begin{equation*}
I_{2,j+v+1}\cup I_{4,j+v+1}\subset I_{1,j+v+1}, 
\end{equation*}
and thus  by Lemma  \ref%
{key1 copy(2)}, we get
\begin{equation}
J_{j,v,l}^{1} \lesssim \left(\frac{R}{H}\right)^{n} H^{\left(\frac{1}{%
\mathbf{p}}+\boldsymbol{\alpha}\right)d} \|f\|_{\dot{E}_{\vec{p}}^{\vec{\alpha},%
\vec{\theta}}(\mathbb{R}^n)}^{d}.  \label{case2}
\end{equation}

\textit{Estimate of }$\tilde{J}_{j,v}^{1}$. Define 
\begin{equation*}
\eta_{R,dN}(\tilde{z}) =R^{n-1}\prod_{i=2}^{n}(1+R|z_i|)^{-dN},\quad 
\tilde{z}=(z_{2},...,z_{n}). 
\end{equation*}
\textit{Case 1.} Let $j\in\mathbb N$ satisfy $2^j>\sqrt n$. As in the
estimate of $J_{j,v,l}^{1}$, we obtain
\begin{align*}
\int_{I_{2,j+v+1}}|f(y)|^{d}\eta _{1,R,dN}(x_{1}-y_{1})dy_{1}& \lesssim
2^{-jdN}R\int_{I_{2,j+v+1}}|f(y)|^{d}dy_{1} \\
& \lesssim R2^{-jdN}2^{(\frac{1}{d}-\frac{1}{p_{1}}-\alpha _{1})(j+v)d}\big\|%
f\big\|_{\dot{K}_{p_{1}}^{\alpha _{1},\infty }}^{d},
\end{align*}%
where the second inequality follows from H\"{o}lder's inequality. Combining
this estimate with Lemma~\ref{key1 copy(2)} and
Remark~\ref{Key2}/(i), we deduce that
\begin{align}
\tilde{J}_{j,v}^{1}& \lesssim R2^{-jdN}2^{(\frac{1}{d}-\frac{1}{p_{1}}%
	-\alpha _{1})(j+v)d}\int_{(I_{4,j+v})^{n-1}}\big\|f\big\|_{\dot{K}%
	_{p_{1}}^{\alpha _{1},\infty }}^{d}\eta _{R,dN}(\tilde{x}-\tilde{y})d\tilde{y%
}  \notag \\
& \lesssim 2^{-j(dN-n+\frac{d}{\mathbf{p}}+d\boldsymbol{\alpha})}\big(\frac{R}{H}\big)%
^{n}H^{(\frac{1}{\mathbf{p}}+\boldsymbol{\alpha})d}\big\|\cdot \cdot \cdot %
\big\|f\big\|_{\dot{K}_{p_{1}}^{\alpha _{1},\infty }}\cdot \cdot \cdot \big\|%
_{\dot{K}_{pn}^{\alpha _{n},\infty }}^{d}.  \label{case3}
\end{align}%
for any $j\in \mathbb{N}$ such that $2^{j}>\sqrt{n}.$

\textit{Case 2.}  Let $j\in\mathbb N$ satisfy $2^j\le \sqrt n$. Arguing as
in Case~2 of the estimate of $J_{j,v,l}^{1}$, we obtain that $\tilde{J}%
_{j,v}^{1}$ can be estimated from above by 
\begin{equation}
c\big(\frac{R}{H}\big)^{n}H^{(\frac{1}{\mathbf{p}}+\boldsymbol{\alpha} )d}\big\|f%
\big\|_{\dot{E}_{\vec{p}}^{\vec{\alpha}_{2},\vec{\theta}}(\mathbb{R}%
	^{n})}^{d}.  \label{case4}
\end{equation}%

Collecting \eqref{case1}--\eqref{case4}, and choosing $N$ sufficiently
large, we obtain the desired result. The proof is complete.
\end{proof}

Let $x_{1}\in \mathbb{R}$, $y=(y_{1},\ldots,y_{n})\in \mathbb{R}^{n}$, and $%
f\in \mathfrak{M}(\mathbb{R}^{n})$. We define 
\begin{equation*}
\eta_{1,R,N} * |f|^{d}(x_{1},\hat{y}_{2}) = \int_{\mathbb{R}}
|f(y)|^{d}\,\eta_{1,R,N}(x_{1}-y_{1})\,dy_{1},
\end{equation*}
that is, the partial convolution with respect to the variable $x_{1}$ (i.e.,
convolution in the first coordinate of $f$).

Let $i\in \{1,\ldots,n-1\}$ and $f\in \mathfrak{M}(\mathbb{R}^{n})$. We
define inductively 
\begin{equation*}
\eta_{i,R,N} * \cdots * \eta_{1,R,N} * |f|^{d}(\check{x}_{i},\hat{y}_{i+1})
= \int_{\mathbb{R}} \eta_{i,R,N}(x_{i}-y_{i})\, \big(\eta_{i-1,R,N} *
\cdots * \eta_{1,R,N} * |f|^{d}\big)(\check{x}_{i-1},\hat{y}_{i})\,dy_{i},
\end{equation*}
where $\check{x}_{i}=(x_{1},\ldots,x_{i})$ and $\hat{y}_{i+1}=(y_{i+1},\ldots,y_{n})$ .

The following lemma can be obtained using the $\dot{K}_{p}^{\alpha,q}$%
-version of the Plancherel-P\'{o}lya-Nikol'skij inequality; see \cite%
{Drihem1.13}.

\begin{lemma}
\label{key1}  Let $\alpha_{1}^{1},\alpha_{1}^{2}\in \mathbb{R}$ and $%
0<q_{1},r_{1}\leq \infty$. Suppose that  
\begin{equation*}
\alpha_{1}^{1}+\frac{1}{s_{1}}>0,\qquad 0<p_{1}\leq s_{1}\leq \infty,\qquad
\alpha_{1}^{2}\geq \alpha_{1}^{1},  
\end{equation*}
and let $f\in \dot{E}_{\vec{p}}^{\vec{\alpha},\vec{\theta}}(\mathbb{R}^{n})$%
. Then we have  
\begin{equation*}
\big\|\big(\eta_{1,R,N} * |f|^{d}(\cdot,\hat{y}_{2})\big)^{1/d}\big\|_{\dot{K%
}_{s_{1}}^{\alpha_{1}^{1},r_{1}}} \lesssim R^{\frac{1}{p_{1}}-\frac{1}{s_{1}}%
+\alpha_{1}^{2}-\alpha_{1}^{1}} \, \|f\|_{\dot{K}_{p_{1}}^{\alpha_{1}^{2},%
\theta_{1}}}(\hat{y}_{2}),
\end{equation*}
for any sufficiently large $N$ and $0<d<\min\!\left(p_{1},\frac{1}{\frac{1}{%
p_{1}}+\alpha_{1}^{2}}\right)$, where  
\begin{equation*}
\theta_{1}= 
\begin{cases}
r_{1}, & \text{if } \alpha_{1}^{2}=\alpha_{1}^{1}, \\[4pt] 
q_{1}, & \text{if } \alpha_{1}^{2}>\alpha_{1}^{1},%
\end{cases}%
\end{equation*}
and the implicit constant is independent of $R$ and $\hat{y}_{2}$.
\end{lemma}

Let $\vec{\alpha}_{1}=(\alpha _{1}^{1},\dots,\alpha
_{n}^{1})\in \mathbb{R}^{n}$,  $\vec{\alpha}_{2}=(\alpha
_{1}^{2},\dots,\alpha _{n}^{2})\in \mathbb{R}^{n}$,  $i \in \{2,\ldots,n-1\}$, $0<\vec{q},\vec{r}\leq \infty$, and $%
(x_{2},\ldots,x_{n}), (y_{2},\ldots,y_{n}) \in \mathbb{R}^{n-1}$. Let $f \in 
\mathfrak{M}(\mathbb{R}^{n})$. We set 
\begin{equation}  \label{def-Vi}
\begin{aligned} V_{i,R,N}(f)(x_{i},\hat{y}_{i+1}) &= \int_{\mathbb{R}}
\big\| \cdots \big\| f\big\|_{\dot{K}_{p_{1}}^{\alpha_{1}^{2},\theta_{1}}}
\cdots \big\|_{\dot{K}_{p_{i-1}}^{\alpha_{i-1}^{2},\theta_{i-1}}}^{d}
(\hat{y}_{i}) \, \eta_{i,R,N}(x_{i}-y_{i})\, dy_{i}. \end{aligned}
\end{equation}

and 
\begin{equation*}
V_{n,R,N}(f)(x_{n}) = \int_{\mathbb{R}} \big\| \cdots \big\| f\big\|_{\dot{K}%
_{p_{1}}^{\alpha_{1}^{2},\theta_{1}}} \cdots \big\|_{\dot{K}%
_{p_{n-1}}^{\alpha_{n-1}^{2},\theta_{n-1}}}^{d} (y_{n}) \,
\eta_{n,R,N}(x_{n}-y_{n})\, dy_{n}.
\end{equation*}

where 
\begin{equation*}
\theta_{i}= 
\begin{cases}
r_{i}, & \text{if } \alpha_{i}^{2}=\alpha_{i}^{1}, \\ 
q_{i}, & \text{if } \alpha_{i}^{2}>\alpha_{i}^{1}.%
\end{cases}%
\end{equation*}

Observe that 
\begin{equation*}
V_{i,R,N}(f)(x_{i},\hat{y}_{i+1}) = \eta_{i,R,N} * \Big( \big\| \cdots %
\big\| f\big\|_{\dot{K}_{p_{1}}^{\alpha_{1}^{2},\theta_{1}}} \cdots \big\|_{%
\dot{K}_{p_{i-1}}^{\alpha_{i-1}^{2},\theta_{i-1}}}^{d} (\hat{y}_{i}) \Big)%
(x_{i}).
\end{equation*}
As in Lemma \ref{key1}, we obtain the following statement.

\begin{lemma}
\label{key1 copy(1)}  Let $i\in \{2,\dots,n\}$, $0<\vec{p},\vec{s},\vec{r},%
\vec{q}\leq \infty$,  $\vec{\alpha}_{1}=(\alpha _{1}^{1},\dots,\alpha
_{n}^{1})\in \mathbb{R}^{n}$,  $\vec{\alpha}_{2}=(\alpha
_{1}^{2},\dots,\alpha _{n}^{2})\in \mathbb{R}^{n}$,  and $f\in \dot{E}_{\vec{%
p}}^{\vec{\alpha},\vec{\theta}}(\mathbb{R}^{n})$.

We suppose that  
\begin{equation*}
\alpha _{i}^{1}+\frac{1}{s_{i}}>0,\qquad 0<p_{i}\leq s_{i}\leq \infty,\qquad
\alpha _{i}^{2}\geq \alpha _{i}^{1}.
\end{equation*}

Then we have 
\begin{equation*}
\big\|(V_{i,R,N}(f)(\cdot ,\hat{y}_{i+1}))^{1/d}\big\|_{\dot{K}
_{s_{i}}^{\alpha _{i}^{1},r_{i}}}\lesssim R^{\frac{1}{p_{i}}-\frac{1}{s_{i}}
+\alpha _{i}^{2}-\alpha _{i}^{1}}\big\|\cdot \cdot \cdot \big\|f\big\|_{\dot{
K}_{p_{1}}^{\alpha _{1}^{2},\theta _{1}}}\cdot \cdot \cdot \big\|_{\dot{K}
_{p_{i}}^{\alpha _{i}^{2},\theta _{i}}}(\hat{y}_{i+1}),
\end{equation*}

for any $N$ large enough and  
$0<d< \min_{1\le i\le n}  \digamma _{i}( p_{i},\alpha _{i}^{2}),$
where $\vec{\theta}=(\theta _{1},\dots,\theta _{n})$ with  
\begin{equation*}
\theta _{i}= 
\begin{cases}
r_{i}, & \text{if } \alpha _{i}^{2}=\alpha _{i}^{1}, \\ 
q_{i}, & \text{if } \alpha _{i}^{2}>\alpha _{i}^{1}.%
\end{cases}%
\end{equation*}
\end{lemma}

The following lemma is the $\dot{K}_{\vec{p}}^{\vec{\alpha},\vec{q}}(\mathbb{%
R}^{n})$-version of the Plancherel--P\'olya--Nikol'skij inequality.

\begin{lemma}
\label{Plancherel-Polya-Nikolskij}  Let $0<\vec{p},\vec{s},\vec{r},\vec{q}%
\leq \infty$,  $\vec{\alpha}_{1}=(\alpha _{1}^{1},\ldots,\alpha _{n}^{1})\in 
\mathbb{R}^{n}$,  $\vec{\alpha}_{2}=(\alpha _{1}^{2},\ldots,\alpha
_{n}^{2})\in \mathbb{R}^{n}$,  and let  $f\in \dot{E}_{\vec{p}}^{\vec{\alpha}%
,\vec{\theta}}(\mathbb{R}^{n})\cap \mathcal{S}^{\prime }(\mathbb{R}^{n})$ 
with  
\begin{equation*}
\mathrm{supp}\,\mathcal{F}f \subset \{\xi \in \mathbb{R}^{n} : |\xi|\leq
R\}.  
\end{equation*}
We suppose that  
\begin{equation*}
\alpha _{i}^{1}+\frac{1}{s_{i}}>0,  \qquad 0<p_{i}\leq s_{i}\leq \infty, 
\qquad \alpha _{i}^{2}\geq \alpha _{i}^{1},  \quad i\in\{1,\ldots,n\}.  
\end{equation*}
Then we have  
\begin{equation*}
\|f\|_{\dot{K}_{\vec{s}}^{\vec{\alpha}_{1},\vec{r}}(\mathbb{R}^{n})} 
\lesssim  R^{\frac{1}{\mathbf{p}}-\frac{1}{\mathbf{s}}+\boldsymbol{\alpha}_{2}-%
\boldsymbol{\alpha}_{1}}  \,  \|f\|_{\dot{E}_{\vec{p}}^{\vec{\alpha}_{2},\vec{%
\theta}}(\mathbb{R}^{n})},  
\end{equation*}
where $\vec{\theta}=(\theta_{1},\ldots,\theta_{n})$ is defined by  
\begin{equation*}
\theta_{i}=  
\begin{cases}
r_{i}, & \text{if } \alpha _{i}^{2}=\alpha _{i}^{1}, \\ 
q_{i}, & \text{if } \alpha _{i}^{2}>\alpha _{i}^{1},%
\end{cases}
\qquad i\in\{1,\ldots,n\}.  
\end{equation*}
\end{lemma}

\begin{proof}
Let $f \in \dot{E}_{\vec{p}}^{\vec{\alpha}_{1},\vec{q}}(\mathbb{R}^{n}) \cap 
\mathcal{S}^{\prime }(\mathbb{R}^{n})$ with $\mathrm{supp}\,\mathcal{F}f
\subset \{\xi \in \mathbb{R}^{n} : |\xi| \leq R\}$. Then, by Lemma \ref%
{r-trick}, we obtain 
\begin{equation*}
|f(x)|^{d} \leq c\, \eta_{R,N} * |f|^{d}(x), \quad x \in \mathbb{R}^{n},
\end{equation*}
where $N$ is chosen sufficiently large and $0<d< \min_{1\le i\le n}  (\digamma _{i}( p_{i},\alpha _{i}^{2}),r_{i})$. Thus, 
\begin{equation*}
\|f\|_{\dot{K}_{\vec{s}}^{\vec{\alpha}_{1},\vec{r}}(\mathbb{R}^{n})}
\lesssim \big\|\eta_{R,N} * |f|^{d}\big\|_{\dot{K}_{\frac{\vec{s}}{d}}^{d%
\vec{\alpha}_{1},\frac{\vec{r}}{d}}}^{1/d},
\end{equation*}
where the implicit constant is independent of $R$. For convenience of the
reader, we decompose the proof into two steps.

\textit{Step 1. } $n=3$. Let $x,y\in \mathbb{R}^{3}$. We have  
\begin{align*}
\eta_{R,N} * |f|^{d}(x) &= \int_{\mathbb{R}}\int_{\mathbb{R}}\int_{\mathbb{R}%
} |f(y)|^{d} \prod_{i=1}^{3}\eta_{i,R,N}(x_i-y_i)\,dy_1dy_2dy_3 \\
&= \int_{\mathbb{R}}\eta_{3,R,N}(x_3-y_3) \int_{\mathbb{R}%
}\eta_{2,R,N}(x_2-y_2)\, \eta_{1,R,N}*|f|^{d}(x_1,y_2,y_3)\,dy_2dy_3 \\
&= \int_{\mathbb{R}}\eta_{3,R,N}(x_3-y_3)\, \big(\eta_{2,R,N} *
(\eta_{1,R,N} * |f|^{d})\big)(x_1,x_2,y_3)\,dy_3 \\
&= \big(\eta_{3,R,N} * \eta_{2,R,N} * \eta_{1,R,N} * |f|^{d}\big)(x).
\end{align*}

Applying Minkowski's inequality with respect to  $\dot{K}_{\frac{s_i}{d}%
}^{d\alpha_i^{1},\,\frac{r_i}{d}}$, $i\in\{1,2\}$, we obtain  
\begin{align*}
\|f\|_{\dot{K}_{\vec{s}}^{\vec{\alpha}_{1},\vec{r}}(\mathbb{R}^{3})}
&\lesssim \Big\|\eta_{3,R,N} * \Big\|\eta_{2,R,N} * \big\|\eta_{1,R,N} *
|f|^{d}\big\|_{\dot{K}_{\frac{s_{1}}{d}}^{d\alpha_{1}^{1},\,\frac{r_{1}}{d}%
}} \Big\|_{\dot{K}_{\frac{s_{2}}{d}}^{d\alpha_{2}^{1},\,\frac{r_{2}}{d}}} %
\Big\|_{\dot{K}_{\frac{s_{3}}{d}}^{d\alpha_{3}^{1},\,\frac{r_{3}}{d}}}^{1/d}.
\end{align*}

By Lemma \ref{key1} and in view of \eqref{def-Vi}, we obtain  
\begin{equation*}
\big\|\eta_{1,R,N} * |f|^{d}\big\|_{\dot{K}_{\frac{s_{1}}{d}%
}^{d\alpha_{1}^{1},\,\frac{r_{1}}{d}}} \lesssim R^{\left(\frac{1}{p_{1}}-%
\frac{1}{s_{1}}+\alpha_{1}^{2}-\alpha_{1}^{1}\right)d} \|f\|_{\dot{K}%
_{p_{1}}^{\alpha_{1}^{2},\theta_{1}}}^{d},
\end{equation*}

and  
\begin{equation*}
\eta_{2,R,N} * \big\|\eta_{1,R,N} * |f|^{d}\big\|_{\dot{K}_{\frac{s_{1}}{d}%
}^{d\alpha_{1}^{1},\,\frac{r_{1}}{d}}} \lesssim R^{\left(\frac{1}{p_{1}}-%
\frac{1}{s_{1}}+\alpha_{1}^{2}-\alpha_{1}^{1}\right)d} V_{2,R,N}(f).
\end{equation*}

By Lemma \ref{key1 copy(1)} and \eqref{def-Vi}, we get  
\begin{equation*}
\|V_{2,R,N}(f)\|_{\dot{K}_{\frac{s_{2}}{d}}^{d\alpha_{2}^{1},\,\frac{r_{2}}{d%
}}} \lesssim R^{\left(\frac{1}{p_{2}}-\frac{1}{s_{2}}+\alpha_{2}^{2}-%
\alpha_{2}^{1}\right)d}
\big\|  \big\| f\big\|_{\dot{K}%
_{p_{1}}^{\alpha_{1}^{2},\theta_{1}}}  \big\|_{\dot{K}%
_{p_{2}}^{\alpha_{2}^{2},\theta_{2}}}^{d}
\end{equation*}

Consequently,  
\begin{equation*}
\|f\|_{\dot{K}_{\vec{s}}^{\vec{\alpha}_{1},\vec{r}}(\mathbb{R}^{3})}
\lesssim R^{\frac{1}{\mathbf{p}}-\frac{1}{\mathbf{s}}+\mathbf{\alpha}_{2}-%
\mathbf{\alpha}_{1}} \|f\|_{\dot{E}_{\vec{p}}^{\vec{\alpha}_{2},\vec{\theta}}(%
\mathbb{R}^{3})}.
\end{equation*}
\textit{Step 2. } Let $n \in \mathbb{N}$. Let $i \in \{1,\dots,n\}$ and $x,y
\in \mathbb{R}^n$.  By induction on $n$, we obtain  
\begin{equation*}
\eta_{R,N} * |f|^{d}(x) = \eta_{n,R,N} * \cdots * \eta_{2,R,N} *
\eta_{1,R,N} * |f|^{d}(x).
\end{equation*}

By the Minkowski inequality in the mixed Herz-type space $\dot{K}_{\frac{%
s_{i}}{d}}^{d\alpha _{i}^{1},\frac{r_{i}}{d}}$, $i \in \{1,\dots,n\}$, we
obtain  
\begin{equation}  \label{main-term-bernstein}
\big\|\eta_{R,N} * |f|^{d}\big\|_{\dot{K}_{\frac{\vec{s}}{d}}^{d\vec{\alpha}%
_1,\frac{\vec{r}}{d}}(\mathbb{R}^n)} \lesssim \Big\|\eta_{n,R,N} * \cdots * %
\Big\|\eta_{2,R,N} * \Big\|\eta_{1,R,N} * |f|^{d}\Big\|_{\dot{K}_{\frac{s_1}{%
d}}^{d\alpha_1^{1},\frac{r_1}{d}}} \Big\|_{\dot{K}_{\frac{s_2}{d}%
}^{d\alpha_1^{2},\frac{r_2}{d}}} \cdots \Big\|_{\dot{K}_{\frac{s_{n}}{d}%
}^{d\alpha_1^{n},\frac{r_n}{d}}}^{1/d} .
\end{equation}

Again, applying induction on $n$ together with Lemma~\ref{key1 copy(1)}, we
conclude that  \eqref{main-term-bernstein} is bounded by  
\begin{equation*}
c\text{ }R^{\frac{1}{\mathbf{p}}-\frac{1}{\mathbf{s}}+\boldsymbol{\alpha}_{2}- 
\boldsymbol{\alpha}_{1}}\big\|\cdot \cdot \cdot \big\|\big\|f\big\|_{\dot{K}
_{p_{1}}^{\alpha _{1}^{2},\theta_{1}}}\big\|_{\dot{K}_{p_{2}}^{\alpha
_{2}^{2},\theta_{2}}}\cdot \cdot \cdot \big\|_{\dot{K}_{p_{n}}^{\alpha
_{n}^{2},\theta_{n}}},
\end{equation*}

This is the desired estimate. The proof is complete.
\end{proof}

\section{The $\varphi$-transform caracterization of 
	$\dot{E}_{\vec{p}}^{\vec{\alpha},\vec{q}}B_{\beta}^{s}$ }

In this section, we introduce the spaces
$\dot{E}_{\vec{p}}^{\vec{\alpha},\vec{q}}B_{\beta}^{s}$  and establish their
$\varphi$-transform characterizations.

Let $\Phi,\varphi\in\mathcal{S}(\mathbb{R}^{n})$ be such that

\begin{equation}
\operatorname{supp}\mathcal{F}\Phi
\subset \{\xi\in\mathbb{R}^{n}:|\xi|\le 2\},
\qquad
|\mathcal{F}\Phi(\xi)|\ge c>0,
\label{Ass1}
\end{equation}
for all $|\xi|\le \frac{5}{3}$, and
\begin{equation}
\operatorname{supp}\mathcal{F}\varphi
\subset
\left\{\xi\in\mathbb{R}^{n}:\frac12\le |\xi|\le 2\right\},
\qquad
|\mathcal{F}\varphi(\xi)|\ge c>0,
\label{Ass2}
\end{equation}
for all $\frac35\le |\xi|\le \frac53$, where $c$ is a positive constant.
Throughout this section, we set
\[
\widetilde{\varphi}(x)=\overline{\varphi(-x)},
\qquad x\in\mathbb{R}^{n}.
\]

We now introduce the function spaces that will be studied in this section.

\begin{definition}
	\label{B-F-def}
	Let $s\in\mathbb{R}$,
	$0<\vec{p},\vec{q}\le\infty$,
	$\vec{\alpha}=(\alpha_{1},\ldots,\alpha_{n})\in\mathbb{R}^{n}$,
	$0<\beta\le\infty$, and let $\Phi$ and $\varphi$ satisfy
	\eqref{Ass1} and \eqref{Ass2}, respectively. For
	$k\in\mathbb{N}$, define
	$\varphi_{k}=2^{kn}\varphi(2^{k}\cdot).$ The \emph{mixed-norm Herz-type Besov space}
	$\dot{E}_{\vec{p}}^{\vec{\alpha},\vec{q}}B_{\beta}^{s}$
	is defined as the collection of all
	$f\in\mathcal{S}'(\mathbb{R}^{n})$ such that
	\[
	\|f\|_{\dot{E}_{\vec{p}}^{\vec{\alpha},\vec{q}}B_{\beta}^{s}}
	=
	\left(
	\sum_{k=0}^{\infty}
	2^{ks\beta}
	\|\varphi_{k}*f\|_{\dot{E}_{\vec{p}}^{\vec{\alpha},\vec{q}}
		(\mathbb{R}^{n})}^{\beta}
	\right)^{1/\beta}
	<\infty,
	\]
	where $\varphi_{0}$ is replaced by $\Phi$, with the usual modification
	when $\beta=\infty$.

\end{definition}

Let $s\in\mathbb{R}$, $0<\vec{p}\le\infty$, and
$0<\beta\le\infty$. Using the system
$\{\varphi_{k}\}_{k\in\mathbb{N}_{0}}$, we define the quasi-norms
\[
\|f\|_{B_{\vec{p},\beta}^{s}}
=
\left(
\sum_{k=0}^{\infty}
2^{ks\beta}
\|\varphi_{k}*f\|_{L^{\vec{p}}(\mathbb{R}^{n})}^{\beta}
\right)^{1/\beta}.
\]

The mixed-norm Besov space $B_{\vec{p},\beta}^{s}$ consists of all
$f\in\mathcal{S}'(\mathbb{R}^{n})$ for which
$\|f\|_{B_{\vec{p},\beta}^{s}}<\infty$.

It is well known that these spaces are independent of the particular
choice of the system
$\{\varphi_{k}\}_{k\in\mathbb{N}_{0}}$, up to equivalence of
quasi-norms. Further details concerning the classical theory of these
spaces, including the homogeneous case, may be found in
\cite{GJN17,JS07,JS08,JHS12}. For the classical Besov spaces, we refer the reader to
\cite{FJ86,FJ90,FrJaWe01,T83,T2}.

It is immediate from the definitions that, when
$\vec{\alpha}=0$ and $\vec{p}=\vec{q}$,
\[
\dot{E}_{\vec{p}}^{0,\vec{p}}B_{\beta}^{s}
=
B_{\vec{p},\beta}^{s},.
\]

Let $\Phi$ and $\varphi$ satisfy \eqref{Ass1} and \eqref{Ass2},
respectively. By \cite[Section 12]{FJ90}, there exist functions
$\Psi\in\mathcal{S}(\mathbb{R}^{n})$ satisfying \eqref{Ass1} and
$\psi\in\mathcal{S}(\mathbb{R}^{n})$ satisfying \eqref{Ass2} such that
\begin{equation}
\mathcal{F}\widetilde{\Phi}(\xi)\mathcal{F}\Psi(\xi)
+
\sum_{k=1}^{\infty}
\mathcal{F}\widetilde{\varphi}(2^{-k}\xi)
\mathcal{F}\psi(2^{-k}\xi)
=
1,
\qquad
\xi\in\mathbb{R}^{n}.
\label{Ass3}
\end{equation}

A fundamental tool in the study of the above spaces is the following
Calder\'{o}n reproducing formula; see \cite[(12.4)]{FJ90} and
\cite[Lemma 2.3]{YSY10}.

\begin{lemma}
	\label{DW-lemma1}
	Let $\Phi,\Psi\in\mathcal{S}(\mathbb{R}^{n})$ satisfy \eqref{Ass1},
	and let $\varphi,\psi\in\mathcal{S}(\mathbb{R}^{n})$ satisfy
	\eqref{Ass2}. Assume that \eqref{Ass3} holds. Then, for every
	$f\in\mathcal{S}'(\mathbb{R}^{n})$,
	\begin{align}
	f
	&=
	\widetilde{\Phi}*\Psi*f
	+\sum_{k=1}^{\infty}
	\widetilde{\varphi}_{k}*\psi_{k}*f
	\notag\\
	&=
	\sum_{m\in\mathbb{Z}^{n}}
	\widetilde{\Phi}*f(m)\,\Psi_{m}
	+
	\sum_{k=1}^{\infty}
	2^{-kn/2}
	\sum_{m\in\mathbb{Z}^{n}}
	\widetilde{\varphi}_{k}*f(2^{-k}m)\,
	\psi_{k,m},
	\label{proc2}
	\end{align}
	with convergence in $\mathcal{S}'(\mathbb{R}^{n})$, where
	\[
	\Psi_{m}=\Psi(\cdot-m),
	\qquad
	\psi_{k,m}
	=
	2^{kn/2}\psi(2^{k}\cdot-m),
	\qquad
	m\in\mathbb{Z}^{n},\ k\in\mathbb{N}.
	\]
\end{lemma}
Let $\Phi,\Psi,\varphi,\psi\in\mathcal{S}(\mathbb{R}^{n})$ satisfy
\eqref{Ass1}--\eqref{Ass3}. The $\varphi$-transform
$S_{\varphi}$ is defined by
\[
(S_{\varphi}f)_{0,m}
=
\langle f,\Phi_{m}\rangle,
\qquad
(S_{\varphi}f)_{k,m}
=
\langle f,\varphi_{k,m}\rangle,
\]
where
\[
\Phi_{m}=\Phi(\cdot-m),
\qquad
\varphi_{k,m}
=
2^{kn/2}\varphi(2^{k}\cdot-m),
\quad
m\in\mathbb{Z}^{n},
\ k\in\mathbb{N}.
\]

The inverse $\varphi$-transform $T_{\psi}$ is defined by
\[
T_{\psi}\lambda
=
\sum_{m\in\mathbb{Z}^{n}}
\lambda_{0,m}\Psi_{m}
+
\sum_{k=1}^{\infty}
\sum_{m\in\mathbb{Z}^{n}}
\lambda_{k,m}\psi_{k,m},
\]
where
$\lambda=\{\lambda_{k,m}\}_{k\in\mathbb{N}_{0},\,m\in\mathbb{Z}^{n}}
\subset\mathbb{C}$; see \cite[p.~131]{FJ90}.

We now introduce the sequence spaces associated with
$\dot{E}_{\vec{p}}^{\vec{\alpha},\vec{q}}B_{\beta}^{s}$.

\begin{definition}
	\label{sequence-space}
	Let
	$s\in\mathbb{R}$,
	$0<\vec{p},\vec{q}\le\infty$,
	$\vec{\alpha}=(\alpha_{1},\ldots,\alpha_{n})\in\mathbb{R}^{n}$,
	and $0<\beta\le\infty$. The space
	$\dot{E}_{\vec{p}}^{\vec{\alpha},\vec{q}}b_{\beta}^{s}$
	is defined as the collection of all complex-valued sequences
	$\lambda=\{\lambda_{k,m}\}_{k\in\mathbb{N}_{0},\,m\in\mathbb{Z}^{n}}$
	such that
	\[
	\|\lambda\|_{\dot{E}_{\vec{p}}^{\vec{\alpha},\vec{q}}b_{\beta}^{s}}
	=
	\left(
	\sum_{v=0}^{\infty}
	2^{v(s+\frac{n}{2})\beta}
	\left\|
	\sum_{m\in\mathbb{Z}^{n}}
	\lambda_{v,m}\chi_{v,m}
	\right\|_{\dot{E}_{\vec{p}}^{\vec{\alpha},\vec{q}}(\mathbb{R}^{n})}^{\beta}
	\right)^{1/\beta}
	<\infty.
	\]

\end{definition}

Observe that the spaces
$\dot{E}_{\vec{p}}^{\vec{\alpha},\vec{q}}B_{\beta}^{s}$
are quasi-normed. If
\[
d=\min\{1,p_{-},q_{-},\beta\},
\]
then
\[
\|f+g\|_{\dot{E}_{\vec{p}}^{\vec{\alpha},\vec{q}}B_{\beta}^{s}}
\lesssim
\|f\|_{\dot{E}_{\vec{p}}^{\vec{\alpha},\vec{q}}B_{\beta}^{s}}
+
\|g\|_{\dot{E}_{\vec{p}}^{\vec{\alpha},\vec{q}}B_{\beta}^{s}}
\]
for all
$f,g\in
\dot{E}_{\vec{p}}^{\vec{\alpha},\vec{q}}B_{\beta}^{s}$.

The following lemma shows that the inverse $\varphi$-transform is
well defined on
$\dot{E}_{\vec{p}}^{\vec{\alpha},\vec{q}}b_{\beta}^{s}$. The proof is very similar as in \cite{DrmixedTriebel}.

\begin{lemma}
	\label{Inv-phi-trans}
	Let
	$s\in\mathbb{R}$,
	$0<\vec{p},\vec{q}\le\infty$,
	$\vec{\alpha}=(\alpha_{1},\ldots,\alpha_{n})\in\mathbb{R}^{n}$,
	$0<\beta\le\infty$,
	and assume that
	\[
	\alpha_{i}>-\frac1{p_{i}},
	\qquad
	i\in \{1,\dots,n\}.
	\]
	Let $\Psi$ and $\psi$ satisfy
	\eqref{Ass1} and \eqref{Ass2}, respectively.
	Then, for every
	$\lambda\in
	\dot{E}_{\vec{p}}^{\vec{\alpha},\vec{q}}b_{\beta}^{s}$,
	the series
	\[
	T_{\psi}\lambda
	=
	\sum_{m\in\mathbb{Z}^{n}}
	\lambda_{0,m}\Psi_{m}
	+
	\sum_{k=1}^{\infty}
	\sum_{m\in\mathbb{Z}^{n}}
	\lambda_{k,m}\psi_{k,m}
	\]
	converges in
	$\mathcal{S}'(\mathbb{R}^{n})$.
	Moreover,
	\[
	T_{\psi}:
	\dot{E}_{\vec{p}}^{\vec{\alpha},\vec{q}}b_{\beta}^{s}
	\longrightarrow
	\mathcal{S}'(\mathbb{R}^{n})
	\]
	is continuous.
\end{lemma}

For a sequence $\lambda =\{\lambda _{k,m}\}_{k\in \mathbb{N}_{0},m\in 
\mathbb{Z}^{n}}\subset \mathbb{C},0<r\leq \infty $ and a fixed $d>0$, set%
\begin{equation*}
\lambda _{k,m,r,d}^{\ast }=\Big(\sum_{h\in \mathbb{Z}^{n}}\frac{|\lambda
_{k,h}|^{r}}{(1+2^{k}|2^{-k}h-2^{-k}m|)^{d}}\Big)^{1/r}
\end{equation*}%
and $\lambda _{r,d}^{\ast }:=\{\lambda _{k,m,r,d}^{\ast }\}_{k\in \mathbb{N}%
_{0},m\in \mathbb{Z}^{n}}\subset \mathbb{C}$ with the usual modification if $%
r=\infty $.

Similarly to the classical Besov spaces (see \cite{FJ90}), we obtain the following estimate.

\begin{lemma}
	\label{lamda-equi copy1}
	Let $s\in \mathbb{R}$, $0<\vec p,\vec q\le \infty$,
	$\vec\alpha=(\alpha_1,\ldots,\alpha_n)\in\mathbb{R}^n$,
	$0<\beta\le\infty$, and $i\in\{1,\ldots,n\}$. Assume that
	\[
	0<a<\min\Bigl(1,p_{-},q_{-},\beta,
	\frac{1}{\alpha_i+\frac1{p_i}}\Bigr),
	\]
	and define
	\[
	r=\min(1,p_{-},q_{-}).
	\]
	If $d>\frac{nr}{a}$, then
	\[
	\lambda_{k,m,r,d}^{*}\,\chi_{k,m}
	\lesssim
	\mathcal{M}_{a}
	\Bigl(
	\sum_{h\in\mathbb Z^{n}}
	\lambda_{k,h}\chi_{k,h}
	\Bigr),
	\qquad
	k\in\mathbb N_{0},\; m\in\mathbb Z^{n}.
	\]
\end{lemma}

From Lemma {\ref{lamda-equi copy1} }and Lemma {\ref{Maximal-Inq}, we obtain
the following statement.}

\begin{lemma}
	\label{lamda-equi}
	Let $s\in \mathbb{R}$, $0<\vec{p},\vec{q}\leq \infty$,
	$\vec{\alpha}=(\alpha_{1},\ldots,\alpha_{n})\in\mathbb{R}^{n}$,
	$0<\beta\leq\infty$, and
	$\alpha_i>-\frac1{p_i}$ for all $i\in\{1,\ldots,n\}$.
	Assume that
	\[
	0<a<
	\min\Bigl(1,p_{-},q_{-},\beta,
	\frac{1}{\alpha_i+\frac1{p_i}}\Bigr),
	\]
	and let
	\[
	r=\min(1,p_{-},q_{-}).
	\]
	If $d>\frac{nr}{a}$, then
	\[
	\bigl\|\lambda_{r,d}^{*}\bigr\|_
	{\dot{E}_{\vec p}^{\vec\alpha,\vec q}b_{\beta}^{s}}
	\approx
	\bigl\|\lambda\bigr\|_
	{\dot{E}_{\vec p}^{\vec\alpha,\vec q}b_{\beta}^{s}}.
	\]
\end{lemma}

We are now in a position to state the following theorem, known as the
$\varphi$-transform characterization in the sense of Frazier and Jawerth.
Its proof follows by a straightforward adaptation of
\cite[Theorem 2.2]{FJ90}, together with Lemma \ref{lamda-equi}.

\begin{theorem}
	\label{phi-tran}
	Let $s\in \mathbb{R}$, $0<\vec{p},\vec{q}\leq \infty$,
	$\vec{\alpha}=(\alpha_{1},\ldots,\alpha_{n})\in\mathbb{R}^{n}$,
	$0<\beta\leq\infty$, and
	$\alpha_i>-\frac1{p_i}$ for all $i\in\{1,\ldots,n\}$.
	Let $\Phi,\Psi\in\mathcal{S}(\mathbb{R}^{n})$ satisfy \eqref{Ass1},
	and let $\varphi,\psi\in\mathcal{S}(\mathbb{R}^{n})$ satisfy
	\eqref{Ass2} such that \eqref{Ass3} holds.
	
	Then the operators
	\[
	S_{\varphi}:
	\dot{E}_{\vec p}^{\vec\alpha,\vec q}B_{\beta}^{s}
	\longrightarrow
	\dot{E}_{\vec p}^{\vec\alpha,\vec q}b_{\beta}^{s}
	\]
	and
	\[
	T_{\psi}:
	\dot{E}_{\vec p}^{\vec\alpha,\vec q}b_{\beta}^{s}
	\longrightarrow
	\dot{E}_{\vec p}^{\vec\alpha,\vec q}B_{\beta}^{s}
	\]
	are bounded. Moreover,
	\[
	T_{\psi}\circ S_{\varphi}
	=
	I
	\]
	on $\dot{E}_{\vec p}^{\vec\alpha,\vec q}B_{\beta}^{s}$.
\end{theorem}

\begin{remark}
	Theorem \ref{phi-tran} provides a powerful tool for studying the spaces
	$\dot{E}_{\vec p}^{\vec\alpha,\vec q}B_{\beta}^{s}$.
	Indeed, many problems can be transferred from the function space to the
	corresponding sequence space, where they are often easier to handle.
	
	More precisely, under the assumptions of Theorem \ref{phi-tran}, one has
	\[
	\bigl\|
	\{\langle f,\varphi_{k,m}\rangle\}_{k\in\mathbb N_{0},\,m\in\mathbb Z^{n}}
	\bigr\|_
	{\dot{E}_{\vec p}^{\vec\alpha,\vec q}b_{\beta}^{s}}
	\approx
	\|f\|_
	{\dot{E}_{\vec p}^{\vec\alpha,\vec q}B_{\beta}^{s}}.
	\]
\end{remark}
Following the arguments in \cite{Dr-AOT23}, we obtain the following result.

\begin{lemma}
	Let $s\in \mathbb{R}$, $0<\vec{p},\vec{q}\leq \infty$, $\vec{\alpha}=(\alpha
	_{1},\ldots,\alpha _{n})\in \mathbb{R}^{n}$, $0<\beta \leq \infty$, and
	$\alpha _{i}>-\frac{1}{p_{i}}$ for all $i\in \{1,\ldots,n\}$. Then the spaces
	$\dot{E}_{\vec{p}}^{\vec{\alpha},\vec{q}}b_{\beta }^{s}$ are quasi-Banach
	spaces. Moreover, they are Banach spaces whenever
	$1\leq \vec{p},\vec{q},\beta \leq \infty$.
\end{lemma}

Combining this lemma with Theorem \ref{phi-tran} and the results of
\cite{GJN17}, we obtain the following useful property of these function
spaces.

\begin{theorem}
	Let $s\in \mathbb{R}$, $0<\vec{p},\vec{q}\leq \infty$, $\vec{\alpha}=(\alpha
	_{1},\ldots,\alpha _{n})\in \mathbb{R}^{n}$, $0<\beta \leq \infty$, and
	$\alpha _{i}>-\frac{1}{p_{i}}$ for all $i\in \{1,\ldots,n\}$. Then the spaces
	$\dot{E}_{\vec{p}}^{\vec{\alpha},\vec{q}}B_{\beta }^{s}$ are quasi-Banach
	spaces. Moreover, they are Banach spaces whenever
	$1\leq \vec{p},\vec{q},\beta \leq \infty$.
\end{theorem}

Let $\vartheta$ be a function in $\mathcal{S}(\mathbb{R}^{n})$ satisfying
\begin{equation}
\vartheta(x)=1 \quad \text{for } |x|\leq 1,
\qquad
\vartheta(x)=0 \quad \text{for } |x|\geq \frac{3}{2}.
\label{function-v}
\end{equation}
Define
\[
\psi_{0}(x)=\vartheta(x), \qquad
\psi_{1}(x)=\vartheta\!\left(\frac{x}{2}\right)-\vartheta(x),
\]
and
\[
\psi_{k}(x)=\psi_{1}(2^{-k+1}x), \qquad k=2,3,\ldots .
\]
Then
\[
\operatorname{supp}\psi_{k}
\subset
\bigl\{x\in\mathbb{R}^{n}:2^{k-1}\leq |x|
\leq 3\cdot 2^{k-1}\bigr\},
\]
and
\begin{equation}
\sum_{k=0}^{\infty}\psi_{k}(x)=1,
\qquad x\in\mathbb{R}^{n}.
\label{partition}
\end{equation}
The family $\{\psi_{k}\}_{k\in\mathbb{N}_{0}}$ is called a smooth dyadic
resolution of unity. Consequently, every
$f\in\mathcal{S}'(\mathbb{R}^{n})$ admits the Littlewood--Paley
decomposition
\[
f=\sum_{k=0}^{\infty}\mathcal{F}^{-1}\psi_{k}*f,
\]
where the series converges in $\mathcal{S}'(\mathbb{R}^{n})$.

We define
\[
\|f\|_{\dot{E}_{\vec{p}}^{\vec{\alpha},\vec{q}}B_{\beta}^{s}}^{\psi_{0},\psi_{1}}
=
\Bigg(
\sum_{k=0}^{\infty}
2^{ks\beta}
\bigl\|
\mathcal{F}^{-1}\psi_{k}*f
\bigr\|_{\dot{E}_{\vec{p}}^{\vec{\alpha},\vec{q}}(\mathbb{R}^{n})}^{\beta}
\Bigg)^{1/\beta}.
\]

The following theorem is proved in the same spirit as the corresponding
result for classical Besov and Triebel--Lizorkin spaces, with the aid of
Lemma \ref{Maximal-Inq}.

\begin{theorem}
	\label{partition-equi}
	Let $s\in \mathbb{R}$, $0<\vec{p},\vec{q}\leq \infty$,
	$\vec{\alpha}=(\alpha _{1},\ldots,\alpha _{n})\in \mathbb{R}^{n}$,
	$0<\beta \leq \infty$, and
	$\alpha _{i}>-\frac{1}{p_{i}}$ for all $i\in \{1,\ldots,n\}$.
	Then a tempered distribution
	$f$ belongs to
	$\dot{E}_{\vec{p}}^{\vec{\alpha},\vec{q}}B_{\beta}^{s}$
	if and only if
	\[
	\|f\|_{\dot{E}_{\vec{p}}^{\vec{\alpha},\vec{q}}B_{\beta}^{s}}^{\psi_{0},\psi_{1}}
	<\infty.
	\]
	Furthermore, the quasi-norms
	\[
	\|f\|_{\dot{E}_{\vec{p}}^{\vec{\alpha},\vec{q}}B_{\beta}^{s}}
	\quad\text{and}\quad
	\|f\|_{\dot{E}_{\vec{p}}^{\vec{\alpha},\vec{q}}B_{\beta}^{s}}^{\psi_{0},\psi_{1}}
	\]
	are equivalent.
\end{theorem}
\section{Embeddings}
The following theorem provides basic embeddings for the spaces
$\dot{E}_{\vec{p}}^{\vec{\alpha},\vec{q}}B_{\beta}^{s}$.

\begin{theorem}
	\label{embeddings1.1}
	Let $s\in\mathbb{R}$, $0<\vec{p},\vec{q}\leq\infty$,
	$\vec{\alpha}=(\alpha_1,\ldots,\alpha_n)\in\mathbb{R}^{n}$,
	$0<\beta\leq\infty$, and
	$\alpha_i>-\frac1{p_i}$ for all $i\in\{1,\ldots,n\}$.
	
	\begin{enumerate}
		\item[(i)] If $0<\beta_1\leq\beta_2\leq\infty$, then
		\begin{equation}
		\dot{E}_{\vec{p}}^{\vec{\alpha},\vec{q}}B_{\beta_1}^{s}
		\hookrightarrow
		\dot{E}_{\vec{p}}^{\vec{\alpha},\vec{q}}B_{\beta_2}^{s}.
		\label{embed1}
		\end{equation}
		
		\item[(ii)] If $0<\beta_1,\beta_2\leq\infty$ and $\varepsilon>0$, then
		\begin{equation}
		\dot{E}_{\vec{p}}^{\vec{\alpha},\vec{q}}B_{\beta_1}^{s+\varepsilon}
		\hookrightarrow
		\dot{E}_{\vec{p}}^{\vec{\alpha},\vec{q}}B_{\beta_2}^{s}.
		\label{embed2}
		\end{equation}
		
		\item[(iii)] Let
		$\vec{q}_1=(q_1^1,\ldots,q_n^1)\in(0,\infty]^n$ and
		$\vec{q}_2=(q_1^2,\ldots,q_n^2)\in(0,\infty]^n$.
		If $0<q_i^1\leq q_i^2\leq\infty$ for all $i\in\{1,\ldots,n\}$, then
		\begin{equation}
		\dot{E}_{\vec{p}}^{\vec{\alpha},\vec{q}_1}B_{\beta}^{s}
		\hookrightarrow
		\dot{E}_{\vec{p}}^{\vec{\alpha},\vec{q}_2}B_{\beta}^{s}.
		\label{embed3}
		\end{equation}
		
		\item[(iv)] Let
		$\vec{p}_1=(p_1^1,\ldots,p_n^1)\in(0,\infty)^n$ and
		$\vec{p}_2=(p_1^2,\ldots,p_n^2)\in(0,\infty)^n$.
		If $0<p_i^1\leq p_i^2\leq\infty$ for all $i\in\{1,\ldots,n\}$, then
		\begin{equation}
		\dot{E}_{\vec{p}_2}^{\vec{\alpha},\vec{q}}B_{\beta}^{s}
		\hookrightarrow
		\dot{E}_{\vec{p}_1}^{\vec{r},\vec{q}}B_{\beta}^{s},
		\label{embed4}
		\end{equation}
		where
		\[
		r_i=\alpha_i-\Big(\frac1{p_i^1}-\frac1{p_i^2}\Big),
		\qquad i\in\{1,\ldots,n\}.
		\]
	\end{enumerate}
\end{theorem}

\begin{proof}
	The embeddings \eqref{embed1} and \eqref{embed3} are immediate consequences
	of the embeddings between Lebesgue sequence spaces and the embedding
	\[
	\dot{E}_{\vec{p}}^{\vec{\alpha},\vec{q}_1}(\mathbb{R}^n)
	\hookrightarrow
	\dot{E}_{\vec{p}}^{\vec{\alpha},\vec{q}_2}(\mathbb{R}^n).
	\]
	
	Let $\Phi$ and $\varphi$ satisfy \eqref{Ass1} and \eqref{Ass2},
	respectively, and let
	$f\in \dot{E}_{\vec{p}}^{\vec{\alpha},\vec{q}}B_{\beta_1}^{s+\varepsilon}$.
	To prove \eqref{embed2}, note that $\varepsilon>0$ implies
	\[
	\Big(
	\sum_{k=0}^{\infty}
	2^{ks\beta_2}
	\big\|\varphi_k*f\big\|_{\dot{E}_{\vec{p}}^{\vec{\alpha},\vec{q}}(\mathbb{R}^n)}^{\beta_2}
	\Big)^{1/\beta_2}
	\leq
	c\sup_{k\in\mathbb{N}_0}2^{k(s+\varepsilon)}
	\big\|
	\varphi_k*f
	\big\|_{\dot{E}_{\vec{p}}^{\vec{\alpha},\vec{q}}(\mathbb{R}^n)}.
	\]
	The desired estimate then follows from the embedding
	$\ell^{\beta_1}\hookrightarrow\ell^{\infty}$.
	Finally, the embedding \eqref{embed4} follows immediately from
	H\"{o}lder's inequality.
\end{proof}

Now we recall the definition of mixed Herz-type Triebel-Lizorkin spaces.
\begin{definition}
	\label{B-F-def}
	Let $s\in\mathbb{R}$,
	$0<\vec{p},\vec{q}<\infty$,
	$\vec{\alpha}=(\alpha_{1},\ldots,\alpha_{n})\in\mathbb{R}^{n}$,
	$0<\beta\le\infty$, and let $\Phi$ and $\varphi$ satisfy
	\eqref{Ass1} and \eqref{Ass2}, respectively. For
	$k\in\mathbb{N}$, define
	$\varphi_{k}=2^{kn}\varphi(2^{k}\cdot).$
	\medskip
	The \emph{mixed-norm Herz-type Triebel--Lizorkin space}
	$\dot{E}_{\vec{p}}^{\vec{\alpha},\vec{q}}F_{\beta}^{s}$
	is defined as the collection of all
	$f\in\mathcal{S}'(\mathbb{R}^{n})$ such that
	\[
	\|f\|_{\dot{E}_{\vec{p}}^{\vec{\alpha},\vec{q}}F_{\beta}^{s}}
	=
	\left\|
	\left(
	\sum_{k=0}^{\infty}
	2^{ks\beta}
	|\varphi_{k}*f|^{\beta}
	\right)^{1/\beta}
	\right\|_{\dot{E}_{\vec{p}}^{\vec{\alpha},\vec{q}}
		(\mathbb{R}^{n})}
	<\infty,
	\]
	where $\varphi_{0}$ is replaced by $\Phi$, with the usual modification
	when $\beta=\infty$.
\end{definition}
\begin{remark}
More properties of mixed Herz-type Triebel-Lizorkin spaces is given in \cite{DrmixedTriebel}.
\end{remark}
The proof of the following theorem is very similar to that for
Herz-type Besov and Triebel--Lizorkin spaces; see \cite{Drihem1.13}.

\begin{theorem}
	\label{embeddings2}
	Let $s\in\mathbb{R}$, $0<\vec{p},\vec{q}\leq\infty$,
	$\vec{\alpha}=(\alpha_1,\ldots,\alpha_n)\in\mathbb{R}^{n}$,
	$0<\beta\leq\infty$, and
	$\alpha_i>-\frac1{p_i}$ for all $i\in\{1,\ldots,n\}$.
	Then
	\[
	\dot{E}_{\vec{p}}^{\vec{\alpha},\vec{q}}
	B_{\min(\beta,p_-,q_-)}^{s}
	\hookrightarrow
	\dot{E}_{\vec{p}}^{\vec{\alpha},\vec{q}}
	F_{\beta}^{s}
	\hookrightarrow
	\dot{E}_{\vec{p}}^{\vec{\alpha},\vec{q}}
	B_{\max(\beta,p_-,q_-)}^{s}.
	\]
\end{theorem}

\begin{remark}
	When $\vec{\alpha}=\vec{0}$ and $\vec{p}=\vec{q}$,
	Theorem~\ref{embeddings2} reduces to the classical embedding theorem for
	mixed-norm Besov and Triebel--Lizorkin spaces.
\end{remark}

\begin{theorem}
	\label{embeddings-S-inf}
	Let $s\in \mathbb{R}$, $0<\vec{p},\vec{q}\leq \infty$,
	$\vec{\alpha}=(\alpha_{1},\ldots,\alpha_{n})\in \mathbb{R}^{n}$,
	$0<\beta\leq\infty$, and
	$\alpha_i>-\frac1{p_i}$ for all $i\in\{1,\ldots,n\}$.
	
	\smallskip
	
	\noindent
	$\mathrm{(i)}$ The embedding
	\begin{equation}
	\mathcal{S}(\mathbb{R}^{n})
	\hookrightarrow
	\dot{E}_{\vec{p}}^{\vec{\alpha},\vec{q}}B_{\beta}^{s}
	\label{embedding}
	\end{equation}
	holds. Moreover, if $0<\vec{p},\vec{q}<\infty$ and
	$0<\beta<\infty$, then $\mathcal{S}(\mathbb{R}^{n})$ is dense in
	$\dot{E}_{\vec{p}}^{\vec{\alpha},\vec{q}}B_{\beta}^{s}$.
	
	\smallskip
	
	\noindent
	$\mathrm{(ii)}$ The embedding
	\begin{equation}
	\dot{E}_{\vec{p}}^{\vec{\alpha},\vec{q}}B_{\beta}^{s}
	\hookrightarrow
	\mathcal{S}'(\mathbb{R}^{n})
	\label{embeddingsSch}
	\end{equation}
	holds.
\end{theorem}

\begin{proof}
	The proof of the embedding \eqref{embedding} and the density of
	$\mathcal{S}(\mathbb{R}^{n})$ in
	$\dot{E}_{\vec{p}}^{\vec{\alpha},\vec{q}}B_{\beta}^{s}$
	follows from the same arguments as in
	\cite{Dr-AOT23,YY1}. Therefore, it remains to prove that
	\[
	\dot{E}_{\vec{p}}^{\vec{\alpha},\vec{q}}B_{\infty}^{s}
	\hookrightarrow
	\mathcal{S}'(\mathbb{R}^{n}).
	\]
	
	Let $\{\varphi_j\}_{j\in\mathbb{N}_0}$ be a smooth dyadic resolution of
	unity. Define
	\[
	\omega_j=\sum_{i=j-1}^{j+1}\varphi_i,
	\qquad j\in\mathbb{N},
	\]
	where $\varphi_{-1}=0$.
	
	Let $f\in \dot{E}_{\vec{p}}^{\vec{\alpha},\vec{q}}B_{\infty}^{s}$ and
	$\psi\in\mathcal{S}(\mathbb{R}^{n})$. Then
	\begin{align*}
	|f(\psi)|
	&\leq
	\sum_{j=0}^{\infty}
	\left|
	(\mathcal{F}^{-1}\varphi_j*f)
	(\mathcal{F}^{-1}\omega_j*\psi)
	\right| \\
	&=
	\sum_{j=0}^{\infty}
	\big\|
	(\mathcal{F}^{-1}\varphi_j*f)
	(\mathcal{F}^{-1}\omega_j*\psi)
	\big\|_{1} \\
	&=
	\sum_{j=0}^{\infty}
	\big\|
	(\mathcal{F}^{-1}\varphi_j*f)
	(\mathcal{F}^{-1}\omega_j*\psi)
	\big\|_{\dot K_{1}^{0,1}(\mathbb{R}%
		^{n})}.
	\end{align*}
	
	Recalling the definition of the spaces $\dot K_{1}^{0,1}$, we obtain
	\begin{align*}
	&\sum_{j=0}^{\infty}
	\sum_{k=-\infty}^{\infty}
	\big\|
	(\mathcal{F}^{-1}\varphi_j*f)
	(\mathcal{F}^{-1}\omega_j*\psi)
	\chi_k
	\big\|_{1} \\
	&\leq
	\sum_{j=0}^{\infty}
	\sum_{k=-\infty}^{\infty}
	\sup_{x\in B(0,2^k)}
	\big|
	\mathcal{F}^{-1}\varphi_j*f(x)
	\big|
	\,
	\big\|
	(\mathcal{F}^{-1}\omega_j*\psi)\chi_k
	\big\|_{1}.
	\end{align*}
	
	We split the latter sum into two parts:
	\[
	I_1=
	\sum_{j=0}^{\infty}
	\sum_{k=-\infty}^{-j-1}
	\sup_{x\in B(0,2^k)}
	\big|
	\mathcal{F}^{-1}\varphi_j*f(x)
	\big|
	\,
	\big\|
	(\mathcal{F}^{-1}\omega_j*\psi)\chi_k
	\big\|_{1},
	\]
	and
	\[
	I_2=
	\sum_{j=0}^{\infty}
	\sum_{k=-j}^{\infty}
	\sup_{x\in B(0,2^k)}
	\big|
	\mathcal{F}^{-1}\varphi_j*f(x)
	\big|
	\,
	\big\|
	(\mathcal{F}^{-1}\omega_j*\psi)\chi_k
	\big\|_{1}.
	\]
	
	Let  $0<d<\min_{1\le i\le n} \digamma _{i}( p_{i},\alpha_{i})$. 
	By Lemma~\ref{Key-est1},
	\begin{align*}
	I_1
	&\leq
	\sum_{j=0}^{\infty}
	\sum_{k=-\infty}^{-j-1}
	2^{(\frac1{\mathbf p}+\boldsymbol{\alpha})j}
	\|
	\mathcal{F}^{-1}\varphi_j*f
	\|_{\dot E_{\vec p}^{\vec\alpha,\vec q}(\mathbb{R}^{n})}
	\,
	\|
	(\mathcal{F}^{-1}\omega_j*\psi)\chi_k
	\|_1 \\
	&\leq
	c
	\|f\|_{\dot E_{\vec p}^{\vec\alpha,\vec q}B_\infty^s}
	\sum_{j=0}^{\infty}
	\sum_{k=-\infty}^{-j-1}
	2^{(\frac1{\mathbf p}+\boldsymbol{\alpha}-s)j}
	\|
	(\mathcal{F}^{-1}\omega_j*\psi)\chi_k
	\|_1 \\
	&\leq
	c
	\|f\|_{\dot E_{\vec p}^{\vec\alpha,\vec q}B_\infty^s}
	\,
	\|\psi\|_{B_{1,1}^{\frac1{\mathbf p}+\boldsymbol{\alpha}-s}}.
	\end{align*}
	
	Applying Lemma~\ref{Key-est1} once more, we obtain
	\begin{align*}
	I_2
	&\leq
	\sum_{j=0}^{\infty}
	\sum_{k=-j}^{\infty}
	2^{j\frac nd}
	2^{(\frac nd-\frac1{\mathbf p}-\boldsymbol{\alpha})k}
	\|
	\mathcal{F}^{-1}\varphi_j*f
	\|_{\dot E_{\vec p}^{\vec\alpha,\vec q}(\mathbb{R}^{n})}
	\,
	\|
	(\mathcal{F}^{-1}\omega_j*\psi)\chi_k
	\|_1 \\
	&\leq
	c
	\|f\|_{\dot E_{\vec p}^{\vec\alpha,\vec q}B_\infty^s}
	\sum_{j=0}^{\infty}
	2^{(\frac nd-s)j}
	\sum_{k=-j}^{\infty}
	2^{(\frac nd-\frac1{\mathbf p}-\boldsymbol{\alpha})k}
	\|
	(\mathcal{F}^{-1}\omega_j*\psi)\chi_k
	\|_1 \\
	&\leq
	c
	\|f\|_{\dot E_{\vec p}^{\vec\alpha,\vec q}B_\infty^s}
	\,
	\|\psi\|_{\dot K_1^{\frac nd-\frac1{\mathbf p}-\boldsymbol{\alpha},1}
		B_1^{\frac nd-s}},
	\end{align*}
	
where $\dot K_1^{\frac nd-\frac1{\mathbf p}-\boldsymbol{\alpha},1}
B_1^{\frac nd-s}$ is the classical Herz-type Besov spaces.	Consequently,
	\[
	|f(\psi)|
	\leq
	c
	\|f\|_{\dot E_{\vec p}^{\vec\alpha,\vec q}B_\infty^s}
	\max\Big(
	\|\psi\|_{B_{1,1}^{\frac1{\mathbf p}+\boldsymbol{\alpha}-s}},
	\,
	\|\psi\|_{\dot K_1^{\frac nd-\frac1{\mathbf p}-\boldsymbol{\alpha},1}
		B_1^{\frac nd-s}}
	\Big).
	\]
	
	By the choice of $d$,
	\[
	\mathcal{S}(\mathbb{R}^{n})
	\hookrightarrow
	\dot K_1^{\frac nd-\frac1{\mathbf p}-\boldsymbol{\alpha},1}
	B_1^{\frac nd-s}.
	\]
	Together with the embedding
	\[
	\mathcal{S}(\mathbb{R}^{n})
	\hookrightarrow
	B_{1,1}^{\frac1{\mathbf p}+\boldsymbol{\alpha}-s},
	\]
	this yields
	\[
	|f(\psi)|
	\leq
	c\,
	\|\psi\|_{\mathcal S_M}
	\,
	\|f\|_{\dot E_{\vec p}^{\vec\alpha,\vec q}B_\infty^s}
	\]
	for some $M\in\mathbb N$.
	
	Hence,
	\[
	\dot E_{\vec p}^{\vec\alpha,\vec q}B_\infty^s
	\hookrightarrow
	\mathcal S'(\mathbb R^n)
	\]
	continuously.

	This completes the proof.
\end{proof}

\section{Sobolev embeddings}

In this section, we establish Sobolev-type embeddings for the spaces
$\dot{E}_{\vec p}^{\vec\alpha,\vec q}B_\beta^s$. It is well known that
\begin{equation}
B_{q,\beta}^{s_{2}}\hookrightarrow B_{p,\beta}^{s_{1}},
\label{Sobolev-emb}
\end{equation}
provided that
\[
s_{1}-\frac{n}{p}=s_{2}-\frac{n}{q},
\]
where $0<q<p\leq \infty$, $s_{1}\leq s_{2}$ and $0<\beta\leq \infty$
(see e.g. \cite[Theorem 2.7.1]{T83}). The following theorem extends
this classical embedding to the scale of mixed-norm Herz--Besov spaces.

\begin{theorem}
	\label{embeddings3}
	Let $s_{1},s_{2}\in\mathbb{R}$,
	$0<\vec{q}\leq \vec{p}\leq \infty$,
	$0<\vec{r},\vec{\nu}\leq \infty$,
	$\vec{\alpha}_{1}=(\alpha_{1}^{1},\ldots,\alpha_{n}^{1})\in\mathbb{R}^{n}$,
	$\vec{\alpha}_{2}=(\alpha_{1}^{2},\ldots,\alpha_{n}^{2})\in\mathbb{R}^{n}$,
	and $0<\beta\leq\infty$. Assume that
	\[
	-\frac{1}{p_{i}}<\alpha_{i}^{1}\leq \alpha_{i}^{2},
	\qquad i\in \{1,\dots,n\},
	\]
	and
	\begin{equation}
	s_{1}-\frac{1}{\mathbf{p}}-\boldsymbol{\alpha}_{1}
	\leq
	s_{2}-\frac{1}{\mathbf{q}}-\boldsymbol{\alpha}_{2}.
	\label{newexp1}
	\end{equation}
	Then
	\begin{equation}
	\dot{E}_{\vec{q}}^{\vec{\alpha}_{2},\vec{\theta}}B_{\beta}^{s_{2}}
	\hookrightarrow
	\dot{E}_{\vec{p}}^{\vec{\alpha}_{1},\vec{r}}B_{\beta}^{s_{1}},
	\label{Sobolev-emb1}
	\end{equation}
	where
	\[
	\vec{\theta}=
	\begin{cases}
	\vec{r}, & \text{if } \vec{\alpha}_{2}=\vec{\alpha}_{1},\\[1mm]
	\vec{\nu}, & \text{if } \vec{\alpha}_{2}>\vec{\alpha}_{1}.
	\end{cases}
	\]
	Moreover, condition \eqref{newexp1} is necessary for the embedding
	\eqref{Sobolev-emb1} to hold.
\end{theorem}
\begin{proof}
	Let $\{\varphi_j\}_{j\in \mathbb{N}_0}$ be a smooth dyadic resolution of
	unity and let
	$f\in \dot{E}_{\vec{q}}^{\vec{\alpha}_2,\vec{\theta}}B_{\beta}^{s_2}$.
	By Lemma \ref{Plancherel-Polya-Nikolskij}, we obtain
	\[
	\bigl\|\mathcal{F}^{-1}\varphi_j * f\bigr\|_
	{\dot{E}_{\vec{p}}^{\vec{\alpha}_1,\vec{r}}(\mathbb{R}^n)}
	\le c\,2^{j\left(
		\frac{1}{\mathbf q}-\frac{1}{\mathbf p}
		-\boldsymbol{\alpha}_1+\boldsymbol{\alpha}_2
		\right)}
	\bigl\|\mathcal{F}^{-1}\varphi_j * f\bigr\|_
	{\dot{E}_{\vec{q}}^{\vec{\alpha}_2,\vec{\theta}}(\mathbb{R}^n)},
	\]
	where $c>0$ is independent of $j\in\mathbb N_0$. The desired embedding
	follows immediately from this estimate and condition \eqref{newexp1}.
	
	We now prove the necessity of \eqref{newexp1}. Let
	$\omega\in\mathcal{S}(\mathbb{R}^n)$ satisfy
	\[
	\operatorname{supp}\mathcal{F}{\omega}
	\subset
	\left\{\xi\in\mathbb{R}^n:\frac34<|\xi|<1\right\}.
	\]
	For $N\in\mathbb N$, define
	\[
	f_N(x)=\omega(2^N x), \qquad x\in\mathbb{R}^n.
	\]
	Clearly,
	\[
	\omega\in
	\dot{E}_{\vec{q}}^{\vec{\alpha}_2,\vec{\theta}}(\mathbb{R}^n)
	\cap
	\dot{E}_{\vec{p}}^{\vec{\alpha}_1,\vec{r}}(\mathbb{R}^n).
	\]
	By the support properties of $\omega$, we have
	\[
	\mathcal{F}^{-1}\varphi_j * f_N=
	\begin{cases}
	f_N, & j=N,\\
	0, & j\neq N.
	\end{cases}
	\]
	
	Let $i\in\{1,\ldots,n\}$. Then
	\begin{align*}
	\|f_N\|_{\dot K_{p_i}^{\alpha_i^1,r_i}}
	&=
	\left(
	\sum_{k_i=-\infty}^{\infty}
	2^{k_i\alpha_i^1 r_i}
	\|f_N\chi_{k_i}\|_{L^{{p_{i}}}(\mathbb R)}^{r_i}
	\right)^{1/r_i}
	\\
	&=
	2^{-N/p_i}
	\left(
	\sum_{k_i=-\infty}^{\infty}
	2^{k_i\alpha_i^1 r_i}
	\|\widetilde{\omega}(2^N\cdot)\chi_{k_i+N}\|_{L^{{p_{i}}}(\mathbb R)}^{r_i}
	\right)^{1/r_i}
	\\
	&=
	2^{-(\alpha_i^1+\frac1{p_i})N}
	\|\widetilde{\omega}(2^N\cdot)\|_{\dot K_{p_i}^{\alpha_i^1,r_i}},
	\end{align*}
	where
	\[
	\tilde{\omega}(2^{N}x)=\omega
	(2^{N}x_{1}...,2^{N}x_{i-1},x_{i},2^{N}x_{i+1},...,2^{N}x_{n}).
	\]
	
	Consequently,
	\[
	\|f_N\|_{\dot{E}_{\vec{p}}^{\vec{\alpha}_1,\vec r}B_\beta^{s_1}}
	=
	2^{s_1N}
	\|f_N\|_{\dot{E}_{\vec{p}}^{\vec{\alpha}_1,\vec r}(\mathbb R^n)}
	=
	2^{\left(
		s_1-\boldsymbol{\alpha}_1-\frac1{\mathbf p}
		\right)N}
	\|\omega\|_
	{\dot{E}_{\vec{p}}^{\vec{\alpha}_1,\vec r}(\mathbb R^n)}.
	\]
	
	Similarly,
	\[
	\|f_N\|_{\dot{E}_{\vec{q}}^{\vec{\alpha}_2,\vec{\theta}}B_\beta^{s_2}}
	=
	2^{\left(
		s_2-\boldsymbol{\alpha}_2-\frac1{\mathbf q}
		\right)N}
	\|\omega\|_
	{\dot{E}_{\vec{q}}^{\vec{\alpha}_2,\vec{\theta}}(\mathbb R^n)}.
	\]
	
	Assume that the embedding \eqref{Sobolev-emb1} holds. Then, for every
	$N\in\mathbb N$,
	\[
	\|f_N\|_{\dot{E}_{\vec{p}}^{\vec{\alpha}_1,\vec r}B_\beta^{s_1}}
	\le
	C\,
	\|f_N\|_{\dot{E}_{\vec{q}}^{\vec{\alpha}_2,\vec{\theta}}B_\beta^{s_2}},
	\]
	which yields
	\[
	2^{\left(
		s_1-s_2
		-\boldsymbol{\alpha}_1+\boldsymbol{\alpha}_2
		-\frac1{\mathbf p}
		+\frac1{\mathbf q}
		\right)N}
	\le C.
	\]
	Letting $N\to\infty$, we conclude that
	\[
	s_1-\boldsymbol{\alpha}_1-\frac1{\mathbf p}
	\le
	s_2-\boldsymbol{\alpha}_2-\frac1{\mathbf q},
	\]
	which is precisely condition \eqref{newexp1}. This completes the proof.
\end{proof}
\begin{remark}
	If $\boldsymbol{\alpha}_{1}=\boldsymbol{\alpha}_{2}=0$ and $\vec{r}=\vec{p}$, then
	Theorem \ref{embeddings3} reduces to the classical embedding results for
	$B_{\vec{p},\beta}^{s}$; see \cite{JS07}. This follows from the embedding
	$\ell^{q_i}\hookrightarrow \ell^{p_i}$ for $i\in\{1,\ldots,n\}$.
\end{remark}

Combining Theorem \ref{embeddings3} with the identity
\[
\dot{E}_{\vec{p}}^{0,\vec{p}}B_{\beta}^{s_{1}}
= B_{\vec{p},\beta}^{s_{1}},
\]
we immediately obtain the following result.

\begin{theorem}
	\label{embeddings4}
	Let $\vec{\alpha}=(\alpha_{1},\ldots,\alpha_{n})\in\mathbb{R}^{n}$,
	$s_{1},s_{2}\in\mathbb{R}$, $0<\vec{r}\leq\infty$,
	$0<\vec{q}\leq\vec{p}\leq\infty$, and $0<\beta\leq\infty$. Assume that
	\[
	s_{1}-\frac{1}{\mathbf p}
	\leq
	s_{2}-\frac{1}{\mathbf q}-\boldsymbol{\alpha}.
	\]
	If
	\[
	\alpha_i\geq0,
	\qquad i\in\{1,\ldots,n\},
	\]
	then
	\[
	\dot{E}_{\vec{q}}^{\vec{\alpha},\vec{\theta}}B_{\beta}^{s_{2}}
	\hookrightarrow
	B_{\vec{p},\beta}^{s_{1}},
	\]
	where
	\begin{equation}
	\vec{\theta}
	=
	\begin{cases}
	\vec{p}, & \text{if }\vec{\alpha}=0,\\
	\vec{r}, & \text{if }\alpha_i>0,\ i\in\{1,\ldots,n\}.
	\end{cases}
	\label{aux10}
	\end{equation}
\end{theorem}

As an immediate consequence, we obtain the following corollary.

\begin{corollary}
	Let $s_{1},s_{2}\in\mathbb{R}$, $0<\vec{q}\leq\vec{p}\leq\infty$,
	$0<\beta\leq\infty$, and
	\[
	s_{1}-\frac{1}{\mathbf p}
	\leq
	s_{2}-\frac{1}{\mathbf q}.
	\]
	Then
	\[
	B_{\vec{q},\beta}^{s_{2}}
	\hookrightarrow
	\dot{E}_{\vec{q}}^{0,\vec{p}}B_{\beta}^{s_{2}}
	\hookrightarrow
	B_{\vec{p},\beta}^{s_{1}}.
	\]
\end{corollary}

\begin{proof}
	It is sufficient to choose $\vec{\theta}=\vec{p}$ and
	$\vec{\alpha}=0$ in Theorem \ref{embeddings4}. Since
	\[
	B_{\vec{q},\beta}^{s_{2}}
	=
	\dot{E}_{\vec{q}}^{0,\vec{q}}B_{\beta}^{s_{2}}
	\hookrightarrow
	\dot{E}_{\vec{q}}^{0,\vec{p}}B_{\beta}^{s_{2}},
	\]
	the asserted embeddings follow immediately.
\end{proof}

The following statement is a direct consequence of Theorem
\ref{embeddings3}.

\begin{theorem}
	\label{embeddings5}
	Let $\vec{\alpha}=(\alpha_{1},\ldots,\alpha_{n})\in\mathbb{R}^{n}$,
	$s_{1},s_{2}\in\mathbb{R}$, $0<\vec{r}\leq\infty$,
	$0<\vec{q}\leq\vec{p}\leq\infty$, and $0<\beta\leq\infty$. Assume that
	\[
	s_{1}-\frac{1}{\mathbf p}-\boldsymbol{\alpha}
	\leq
	s_{2}-\frac{1}{\mathbf q}.
	\]
	If
	\[
	-\frac1{p_i}<\alpha_i\leq0,
	\qquad i\in\{1,\ldots,n\},
	\]
	then
	\[
	B_{\vec{q},\beta}^{s_{2}}
	\hookrightarrow
	\dot{E}_{\vec{p}}^{\vec{\alpha},\vec{\theta}}B_{\beta}^{s_{1}},
	\]
	where
	\[
	\vec{\theta}
	=
	\begin{cases}
	\vec{q}, & \text{if }\vec{\alpha}=0,\\
	\vec{r}, & \text{if }-\dfrac1{p_i}<\alpha_i<0,\quad
	i\in\{1,\ldots,n\}.
	\end{cases}
	\]
\end{theorem}

As a consequence, we obtain the following corollary.

\begin{corollary}
	Let 
	$s_{1},s_{2}\in\mathbb{R}$, $0<\vec{r}\leq\infty$,
	$0<\vec{q}\leq\vec{p}\leq\infty$, and $0<\beta\leq\infty$. Assume that
	\[
	s_{1}-\frac{1}{\mathbf p}
	\leq
	s_{2}-\frac{1}{\mathbf q}.
	\]
	Then
	\[
	B_{\vec{q},\beta}^{s_{2}}
	\hookrightarrow
	\dot{E}_{\vec{p}}^{0,\vec{q}}B_{\beta}^{s_{1}}
	\hookrightarrow
	B_{\vec{p},\beta}^{s_{1}}.
	\]
\end{corollary}

\begin{proof}
	It is sufficient to choose $\vec{\theta}=\vec{q}$ and
	$\vec{\alpha}=0$ in Theorem \ref{embeddings5}. Since
	\[
	\dot{E}_{\vec{p}}^{0,\vec{q}}B_{\beta}^{s_{1}}
	\hookrightarrow
	\dot{E}_{\vec{p}}^{0,\vec{p}}B_{\beta}^{s_{1}}
	=
	B_{\vec{p},\beta}^{s_{1}},
	\]
	the conclusion follows.
\end{proof}

\begin{remark}
	The Sobolev-type embeddings established in this section extend and refine
	the corresponding results obtained in \cite{JS07}.
\end{remark}

\textbf{Data Availability Statement} Data sharing is not applicable to this article as no
data sets were generated or analyzed during this study.

\section*{Declarations}

\textbf{Funding and/or Conflicts of interests/Competing interests:} This work was supported by the General Direction of Higher Education and Training under Grant No. C00L03UN280120220004 and by the General Directorate of Scientific Research and Technological Development, Algeria. The author declares that there are no conflicts of interest or competing interests

\end{document}